\documentclass[11pt,a4paper,reqno]{amsart}

\usepackage{times} % assumes new font selection scheme installed
\usepackage{amsmath} % assumes amsmath package installed
\usepackage{amssymb}  % assumes amsmath package installed
\usepackage{amsthm}
\usepackage{latexsym}
\usepackage{amsfonts,bbm}
\usepackage{xcolor}
\usepackage{mathtools}
\usepackage{enumerate}
\usepackage{cite}
\usepackage{hyperref}

\usepackage{tikz}
\usepackage{pgfplots}
\usetikzlibrary{math,patterns,shapes.geometric,arrows.meta}
\usepackage{graphicx}
\usepackage{todonotes}
\usepackage{appendix}

\newtheorem{thm}{Theorem}[section]

\newtheorem{lem}[thm]{Lemma}
\newtheorem{prop}[thm]{Proposition}
\theoremstyle{definition}
\newtheorem{defn}[thm]{Definition}

\newtheorem{rem}[thm]{Remark}

\numberwithin{equation}{section}
\allowdisplaybreaks
\newcommand{\sbvek}[2]{\left[\begin{smallmatrix}#1\\#2\end{smallmatrix}\right]}
\newcommand{\spvek}[2]{\left(\begin{smallmatrix}#1\\#2\end{smallmatrix}\right)}
\newcommand{\sbmat}[4]{\left[\begin{smallmatrix}#1 & #2\\#3 & #4\end{smallmatrix}\right]}

\newcommand{\divg}{\operatorname{div}}

\newcommand{\dom}{\operatorname{dom}}

\DeclareMathOperator{\supp}{supp}

\newcommand{\frakA}{\mathfrak A}

\newcommand{\frakC}{\mathfrak C}

\newcommand{\frakS}{\mathfrak S}
\newcommand{\Id}{I}%\operatorname{Id}}
\newcommand{\setdef}[2]{\left\{ #1 \left\vert\vphantom{#1} #2 \right.\right\}}

\newcommand{\N}{\mathbb{N}}
\newcommand{\Z}{\mathbb{Z}}
\newcommand{\R}{\mathbb{R}}

\newcommand{\<}{\langle}
\renewcommand{\>}{\rangle}
\newcommand{\loc}{\mathrm{loc}}

\title[modulating function method for infinite-dimensional systems]{The modulating function method for state estimation and feedback
of infinite-dimensional systems}

\author[F.\ Friedrich]{Folke Friedrich}
\address{{\bf F.~Friedrich:} Control Engineering group, Technische Universit\"at Ilmenau, Ilmenau, Germany}
\email{folke.friedrich@tu-ilmenau.de}
\author[J.\ Reger]{Johann Reger}
\address{{\bf J.~Reger:} Control Engineering group, Technische Universit\"at Ilmenau, Ilmenau, Germany}
\email{johann.reger@tu-ilmenau.de}
\author[T.\ Reis]{Timo Reis}
\address{{\bf T.~Reis:} Systems Theory and Partial Differential Equations group, Technische Universit\"at Ilmenau, Ilmenau, Germany}
\email{timo.reis@tu-ilmenau.de}
\begin{document}
\begin{abstract} We investigate state estimation and feedback for infinite-dimensional linear systems, including systems governed by partial differential equations with boundary control and observation. We extend the modulating function approach to infinite-dimensional systems. The basic idea is to reconstruct selected components of the state by convolving the input and output with null controls and corresponding outputs of the adjoint system. We further develop a distributional solution framework for the adjoint system that allows us to treat unbounded feedback operators. In particular, this framework enables the realisation of feedback involving spatial point evaluations for partial differential equations. \end{abstract}

\maketitle



\bigskip

\section{Introduction}\label{sec:intro}

The modulating function approach reconstructs selected components of the state by convolving the input and output with suitably chosen precomputed functions. More precisely, for a~linear system
\begin{equation}
    \begin{aligned}
    \dot{x}(t)&=Ax(t)+Bu(t),\\
    y(t)&=Cx(t)+Du(t),
\end{aligned}\label{eq:ABCDfin}
\end{equation}
$A\in\R^{n\times n}$, $B\in\R^{n\times m}$,
$C\in\R^{p\times n}$, $D\in\R^{p\times m}$, the approach consists of constructing, for given $T>0$ and $\varphi_0\in\R^n$, functions $\mu:[0,T]\to\R^m$ and $\eta:[0,T]\to\R^p$ such that every trajectory of \eqref{eq:ABCDfin} satisfies
\begin{equation}
    \forall\,t\geq T: \quad \varphi_0^\top x(t)=\int_{0}^T u(t-\tau)^\top \mu(\tau)-y(t-\tau)^\top \eta(\tau){\rm d}\tau.
\label{eq:movhor_fin}\end{equation}
Thus, $\varphi_0^\top x(t)$, which represents a component of $x(t)$ in the direction determined by $\varphi_0$, can be reconstructed from the input and output on the moving horizon $[t-T,t]$ by evaluating the integral in \eqref{eq:movhor_fin}. A straightforward integration by parts shows that suitable functions $\mu$ and $\eta$ are obtained by solving the following null-control problem for the adjoint system
\begin{equation}
    \begin{aligned}
    \dot{\varphi}(t)&=A^\top \varphi(t)+C^\top \eta(t),\qquad \varphi(0)=\varphi_0,\;\;\;\varphi(T)=0,\\
    \mu(t)&=B^\top \varphi(t)+D^\top\eta(t).
\end{aligned}\label{eq:ABCDfin_dualnull}
\end{equation}
Then $\mu$ and $\eta$ satisfy the reconstruction identity \eqref{eq:movhor_fin}. Classical systems theory implies that the control problem \eqref{eq:ABCDfin_dualnull} is solvable if the system \eqref{eq:ABCDfin} is observable. Moreover, the solution of \eqref{eq:ABCDfin_dualnull} can be computed offline, so that the online reconstruction requires only the evaluation of the convolution integrals.
This approach can be used for two purposes:
\begin{enumerate}[(a)]
    \item {\bf State estimation:} By choosing a linearly independent family $(\varphi_{01},\ldots,\varphi_{0n})$ and determining corresponding null controls $\eta_1,\ldots,\eta_n$ and outputs $\mu_1,\ldots,\mu_n$ of the adjoint system, we can reconstruct the scalars $\varphi_{0i}^\top x(t)$ for $i=1,\ldots,n$. These values then allow us to determine $x(t)$. If only $N\leq n$ such null controls are used, we obtain the orthogonal projection, and hence the best approximation, of $x(t)$ onto the $N$-dimensional subspace $\operatorname{span}(\varphi_{01},\ldots,\varphi_{0N})$. The reconstruction is particularly simple if the initial values of the adjoint null-control problems form an orthonormal family.
\item {\bf State feedback:} Here, the aim is to realise the state feedback $u(t)=Fx(t)$, $F\in\R^{m\times n}$, using only the measured input and output. To achieve this, we solve $m$ null control problems for the adjoint system, where the initial values $\varphi_{01},\ldots,\varphi_{0m}\in\R^n$ are the transposes of the row vectors of $F$. Then \eqref{eq:movhor_fin} shows that the moving-horizon integrals of the input and output of \eqref{eq:ABCDfin} yield %an expression for
\[\left(\begin{smallmatrix}
\varphi_{01}^\top x(t)\\[-1.2mm]\vdots\\
\varphi_{0m}^\top x(t)
\end{smallmatrix}\right)=\left[\begin{smallmatrix}
\varphi_{01}^\top\\[-1.2mm]\vdots\\
\varphi_{0m}^\top
\end{smallmatrix}\right]x(t)=Fx(t).\]
    \end{enumerate}
Although the modulating-function approach is relatively straightforward for finite-dimensional systems, its extension to infinite-dimensional systems involves substantial difficulties, particularly in the presence of boundary control or observation. To cover a broad class of infinite-dimensional linear systems, we employ the system-node framework developed by {\sc Staffans} in \cite{Staffans2005} and related publications. Such systems are described by
\begin{equation}
\spvek{\dot{x}(t)}{y(t)}
= \sbvek{A\&B\\[-1mm]}{C\&D} \spvek{{x}(t)}{u(t)},\label{eq:ODEnode}\end{equation}
where $A\&B:X\times U\supset\dom(A\&B)\to X$ and $C\&D:X\times U\supset\dom(C\&D)\to Y$ are linear operators whose precise properties are specified in the next section. The symbols $\&$ indicate that the domains of these operators need not be Cartesian products of subspaces of $X$ and $U$. Although the comprehensive theory of system nodes developed in \cite{Staffans2005} contains only a~few examples involving partial differential equations, many such systems can be incorporated naturally into this framework, including systems on multidimensional spatial domains. Examples include systems governed by the wave equation, advection--diffusion equations, Maxwell's equations, the Oseen equations, and the Euler--Bernoulli and Timoshenko beam equations \cite{ReSc23a,PhReSc23,FJRS24,ReSc24}.

A further important advantage is that the class of system nodes is closed under adjunction, which allows us to formulate an infinite-dimensional counterpart of the null-control problem \eqref{eq:ABCDfin_dualnull}.

The objective of this article is to develop a modulating-function framework for infinite-dimensional linear systems described by system nodes, with particular emphasis on state estimation and the realisation of state feedback using only input and output data.
For state estimation, we show that, for $\varphi_0\in X$ and suitable functions $\mu\in L^2([0,T];U)$ and $\eta\in L^2([0,T];Y)$ arising from a null-control problem for the adjoint system,
\begin{equation}
    \forall\,t\geq T: \quad \langle x(t),\varphi_0\rangle_X=\int_{0}^T \langle u(t-\tau),\mu(\tau)\rangle_U -\langle y(t-\tau),\eta(\tau)\rangle_Y{\rm d}\tau.
\label{eq:movhor_inf}\end{equation}
Applying this identity simultaneously to initial values $\varphi_{01},\ldots,\varphi_{0k}$ of the adjoint null-control problem yields the orthogonal projection of $x(t)$ onto $\operatorname{span}(\varphi_{01},\ldots,\varphi_{0k})$.

The application to state feedback becomes more involved when the feedback operator is unbounded. More precisely, we consider feedback laws of the form $u=Fx$, where $F$ maps an intermediate space $\mathcal{V}$ between $\dom(A)$ and $X$ into the input space $U$. In this setting, $U$ is assumed to be finite-dimensional. For $U=\R^m$, the feedback operator $F$ can be represented by an $m$-tuple of elements of $\mathcal{V}^\prime$,
%\begin{equation}F=\left[\begin{smallmatrix}F_1\\\vdots\\F_m\end{smallmatrix}\right]\label{eq:Ftuple}\end{equation}
%with $F_1,\ldots,F_m\in\mathcal{V}^\prime$,
where $\mathcal{V}^\prime$ denotes the dual of $\mathcal{V}$ with respect to the pivot space $X$. The space $\mathcal{V}^\prime$ may be regarded as an extension of the state space of the adjoint system. We therefore seek inputs $\eta_i$ that drive the adjoint system, formally initialised at $\varphi(0)=F_i$, to zero in an appropriate sense. Such null controls need not exist within a classical function-space setting. This motivates the extension of the solution theory for system nodes to distributional inputs and trajectories. In this distributional setting, the reconstruction identity \eqref{eq:movhor_inf} requires a more general interpretation. More precisely, the action of $F_i$ on $x(t)$ is represented by convolutions of the distributions $\eta_i$ and $\mu_i$ with the input $u$ and the output $y$, rather than by moving-horizon $L^2$-inner products. The treatment of unbounded feedback operators is particularly important because it allows feedback laws involving spatial point evaluations for systems governed by partial differential equations.

We do not assume that the system or its adjoint is well-posed. Although well-posedness considerably simplifies the analysis, it restricts the class of systems under consideration, and suitable well-posedness criteria are not available for all systems of practical interest. By allowing non-well-posed system nodes, our framework applies to a broader class of boundary-controlled partial differential equations. This generality, together with the treatment of unbounded feedback operators, requires a more extensive theoretical framework.

The remainder of this article is organised as follows. Section~\ref{sec:prelim} introduces the notation, function spaces, and operator-theoretic framework used throughout the article. In Section~\ref{sec:sysnode}, we recall the basic theory of system nodes and develop a distributional solution framework. Section~\ref{sec:dual} considers the adjoint system and its systems-theoretic properties and contains our main result on distributional partial-state reconstruction. Section~\ref{sec:stateest} applies this result to state estimation, while Section~\ref{sec:feedback} addresses the implementation of state feedback. Section~\ref{sec:ex} presents two examples. The first concerns the construction of an exponentially stabilising feedback law for a vibrating string with boundary force control and an interior velocity measurement. The second considers state estimation for a reaction--diffusion equation on a two-dimensional spatial domain with Dirichlet boundary control and Neumann boundary observation.

We conclude this introduction with a brief overview of the literature on the modulating-function approach. The method originates from the work of \textsc{Shinbrot} \cite{Shin1954}. Its basic idea is to use \emph{modulating functions} to transform linear and certain nonlinear ordinary differential equations into algebraic equations, thereby enabling the application of regression methods to the estimation of unknown system parameters. In the following decades, the modulating-function method was applied to parameter estimation for increasingly broad classes of systems. Of particular relevance to the present work are the early applications to systems governed by partial differential equations by \textsc{Perdreauville} and \textsc{Goodson} \cite{PerdG1966} and, subsequently, by \textsc{Fairman} and \textsc{Shen} \cite{FairS1970}. Subsequent developments largely focused on finite-dimensional systems and, in particular, on the choice of suitable modulating functions. The modulating-function method was first applied to state estimation for finite-dimensional systems in \cite{LiuLKPG2014,JoufR2015}, where modulating functions with non-zero boundary values were employed. Since the original work of \textsc{Shinbrot}, modulating functions have typically been required only to satisfy suitable homogeneous boundary conditions and differentiability assumptions, leaving their precise shape as a degree of freedom.

In \cite{FiscD2016}, \textsc{Fischer} and \textsc{Deutscher} developed a modulating-function-based fault-detection framework for systems governed by partial differential equations. Their approach transforms the differential equation into an algebraic relation that gives rise to a \emph{kernel equation}, whose solution determines the modulating function. This approach was subsequently used in \cite{GhafNRLK2020} to construct a state estimator for parabolic partial differential equations. The resulting estimate is given in terms of the coefficients of a truncated series expansion of the distributed state. The approach was extended to coupled parabolic systems in \cite{RojaNRPZ2022} and to parabolic systems with certain nonlinear reaction terms in \cite{GhafNRLK2023}. To the best of our knowledge, a systematic and unified modulating-function framework for infinite-dimensional systems that accommodates boundary control and unbounded feedback operators has not yet been developed.

\section{Notation, function spaces, and operator-theoretic preliminaries}\label{sec:prelim}

The sets of positive and non-negative integers are denoted by $\mathbb{N}$ and $\mathbb{N}_0$, respectively. Unless stated otherwise, all Hilbert spaces considered in this article are real. By complexification, the results can also be formulated for complex Hilbert spaces. The norm on $X$ is denoted by $\|\cdot\|_X$, or simply by $\|\cdot\|$ if the underlying space is clear from the context. The identity operator on $X$ is denoted by $\Id_X$, or simply by $\Id$.

The topological dual of $X$ is denoted by $X^\prime$, and the corresponding duality pairing by $\<\cdot,\cdot\>_{X^\prime,X}$. Using the reflexivity of $X$, we set
$\<x,x^\prime\>_{X,X^\prime}\coloneqq \<x^\prime,x\>_{X^\prime,X}$ for all $x\in X$ and $x^\prime\in X^\prime$.

The space of bounded linear operators from $X$ to $Y$ is denoted by $L(X,Y)$, and we write $L(X)\coloneqq L(X,X)$. The domain $\dom A$ of a possibly unbounded linear operator $A:X\supset\dom A\to Y$ is endowed with the graph norm
\[
\|x\|_{\dom A}\coloneqq \big(\|x\|_{X}^2+\|Ax\|_{Y}^2\big)^{1/2}.
\]
For a densely defined linear operator $A:X\supset\dom A\to Y$, its dual operator $A^\prime:Y^\prime\supset\dom A^\prime\to X^\prime$ is defined by
\[
\dom A^\prime=\setdef{y^\prime\in Y^\prime}{\exists\, z^\prime\in X^\prime\text{ s.t.\ }\forall\,x\in\dom A:\;\langle Ax,y^\prime\rangle_{Y,Y^\prime}=\langle z^\prime,x\rangle_{X^\prime,X}}.
\]
For each $y^\prime\in\dom A^\prime$, the element $z^\prime\in X^\prime$ in the above definition is uniquely determined, and we set $A^\prime y^\prime\coloneqq z^\prime$. If $X$ and $Y$ are identified with their dual spaces, the dual operator is called the {\em adjoint of $A$} and is denoted by $A^*:Y\supset\dom A^*\to X$.

For Lebesgue and Sobolev spaces, we follow the notation of {\sc Adams} \cite{adams2003sobolev}. For spaces of functions taking values in a Hilbert space $X$, the target space is indicated after a semicolon. For example, the space of $p$-integrable $X$-valued functions on $\Omega$ is denoted by $L^p(\Omega;X)$.

For $T>0$ and $k\in\N$, we define
\begin{align}
    H^k_{0l}([0,T];X) &\coloneqq  \setdef{v\in H^k([0,T];X)}{v(0)= \cdots =\tfrac{\mathrm{d}^{k-1}}{\mathrm{d}t^{k-1}}v(0) = 0},\label{eq:Hkl}\\
    H^k_{0r}([0,T];X) &\coloneqq  \setdef{v\in H^k([0,T];X)}{v(T)= \cdots =\tfrac{\mathrm{d}^{k-1}}{\mathrm{d}t^{k-1}}v(T) = 0}.\label{eq:Hkr}
\end{align}
These spaces will be used throughout the article. Both are closed subspaces of $H^k([0,T];X)$ and hence Hilbert spaces with respect to the norm inherited from $H^k([0,T];X)$.

Using $L^2([0,T];X)$ as the pivot space, we define the corresponding negative-order spaces by
\begin{equation}
    H^{-k}_{0l}([0,T];X) \coloneqq  H^k_{0r}([0,T];X)^\prime,\qquad
    H^{-k}_{0r}([0,T];X) \coloneqq  H^k_{0l}([0,T];X)^\prime.
\label{eq:H-kl}
\end{equation}
For $k=0$, we set
\[
H^{0}_{0l}([0,T];X)\coloneqq L^2([0,T];X)=:H^{0}_{0r}([0,T];X).
\]
The Dirac distribution $\delta\in H^{-1}_{0l}([0,T])$ is defined by
\[
\langle\delta,w\rangle_{H^{-1}_{0l}([0,T]),H^{1}_{0r}([0,T])}
=w(0)
\qquad
\forall\,w\in H^{1}_{0r}([0,T]).
\]
For $k\geq0$, the {\em left derivative}
\[
\big(\tfrac{\mathrm{d}}{\mathrm{d}t}\big)_l:
H^{-k}_{0l}([0,T];X)\to H^{-k-1}_{0l}([0,T];X)
\]
is defined as the negative dual of the right derivative, that is,
\begin{align}\label{eq:l2nder}
\begin{split}
   \left\langle \big(\tfrac{\mathrm{d}}{\mathrm{d}t}\big)_l v, w\right\rangle_{H^{-k-1}_{0l}([0,T];X),H^{k+1}_{0r}([0,T];X)}
    &\coloneqq -\left\langle v,\tfrac{\mathrm{d}}{\mathrm{d}t}w\right\rangle_{H^{-k}_{0l}([0,T];X),H^{k}_{0r}([0,T];X)}
    \\
    &\hspace{-2cm}\forall\,v\in H^{-k}_{0l}([0,T];X),\,
    w\in H^{k+1}_{0r}([0,T];X).
\end{split}
\end{align}
For $v\in H^{k}_{0l}([0,T];X)$, $k\in\N$, the left derivative agrees with the conventional weak derivative. Consequently, for every $k\in\Z$,
$\big(\tfrac{\mathrm{d}}{\mathrm{d}t}\big)_l$ is a bounded bijective operator from
$H^{k}_{0l}([0,T];X)$ onto $H^{k-1}_{0l}([0,T];X)$.
The \emph{right derivative}
\[
\big(\tfrac{\mathrm{d}}{\mathrm{d}t}\big)_r:
H^{k}_{0r}([0,T];X)\to H^{k-1}_{0r}([0,T];X),
\qquad k\in\Z,
\]
is defined analogously.

For $\tau\in[0,T]$, the {\em $\tau$-right shift operator}
$S_{r,\tau}\in L(L^2([0,T];X))$ maps $v\in L^2([0,T];X)$ to the function $S_{r,\tau}v$ defined by
\[
(S_{r,\tau}v)(t)=
\begin{cases}
0,&t\in[0,\tau),\\
v(t-\tau),&t\in[\tau,T],
\end{cases}
\]
for almost every $t\in[0,T]$. Its restriction defines a bounded operator on $H^k_{0l}([0,T];X)$. Analogously, the {\em $\tau$-left shift operator} $S_{l,\tau}$ defines a bounded operator on $H^k_{0r}([0,T];X)$. For $k\geq0$, the $\tau$-right shift on $H^{-k}_{0l}([0,T];X)$ is defined as the dual of the $\tau$-left shift on $H^{k}_{0r}([0,T];X)$. With this definition, the right shift on $L^2([0,T];X)$ extends to the spaces $H^{-k}_{0l}([0,T];X)$.

Let $A\in L(X,Y)$ and $k\in\N$. For $w\in H^{k}_{0l}([0,T];X)$, the pointwise application
$Aw\coloneqq \big(t\mapsto Aw(t)\big)$ belongs to $H^{k}_{0l}([0,T];Y)$. Likewise, $Aw\in H^{k}_{0r}([0,T];Y)$ whenever $w\in H^{k}_{0r}([0,T];X)$. The pointwise application of $A$ to an element $v\in H^{-k}_{0l}([0,T];X)$ is defined by
\begin{multline}\label{eq:Apointwl}
    \langle A v, w\rangle_{H^{-k}_{0l}([0,T];Y),H^k_{0r}([0,T];Y)}
    \coloneqq \langle v, A^* w\rangle_{H^{-k}_{0l}([0,T];X),H^k_{0r}([0,T];X)}
    \\
    \forall\,v\in H^{-k}_{0l}([0,T];X),\,
    w\in H^k_{0r}([0,T];Y).
\end{multline}
The pointwise application of $A$ to elements of $H^{-k}_{0r}([0,T];X)$ is defined analogously. With these definitions, the pointwise application of $A$ commutes with both the left and the right derivative, that is,
\[
\big(\tfrac{{\rm d}}{{\rm d}t}\big)_lA
=A\big(\tfrac{{\rm d}}{{\rm d}t}\big)_l,
\qquad
\big(\tfrac{{\rm d}}{{\rm d}t}\big)_rA
=A\big(\tfrac{{\rm d}}{{\rm d}t}\big)_r.
\]

We next introduce function spaces on the infinite time horizon. The elements of the spaces defined above can be canonically identified with distributions in the sense of \textsc{Schwartz}. More precisely, an element of $H^{-k}_{0l}([0,T];X)$, $k\in\N$, defines a distribution on $(-\infty,T)$ supported in $[0,T)$ by
\[
\langle v,\psi\rangle
\coloneqq 
\langle v,\psi|_{[0,T]}\rangle_{H^{-k}_{0l}([0,T];X),H^k_{0r}([0,T];X)},
\qquad
\psi\in C^\infty_0((-\infty,T);X).
\]
The elements of $H^{-k}_{0r}([0,T];X)$ are identified analogously with distributions on $(0,\infty)$ supported in $(0,T]$. Elements of non-negative order are regarded as regular distributions.

Accordingly, the support of an element of any of these spaces is understood as its distributional support in the sense of \textsc{Schwartz}; see \cite[Def.~6.22]{Rudi73}. Equivalently, for $v\in H^{-k}_{0l}([0,T];X)$, a relatively open set $J\subset[0,T)$ is disjoint from $\supp v$ if and only if
\[
\langle v,w\rangle_{H^{-k}_{0l}([0,T];X),H^k_{0r}([0,T];X)}
=0
\]
for every $w\in H^k_{0r}([0,T];X)$ with $\supp w\subset J$. The analogous characterisation applies to elements of $H^{-k}_{0r}([0,T];X)$.

For $k\in\N_0$, the space $H^k_{0,\loc}(\R_{\ge0};X)$ consists of all $v\in L^2_{\loc}(\R_{\ge0};X)$ such that, for every $T>0$, the restriction of $v$ to $[0,T]$ belongs to $H^k_{0l}([0,T];X)$. We equip this space with the locally convex topology generated by the seminorms given by the $H^k_{0l}([0,T];X)$-norms of these restrictions; see \cite[Chap.~1]{Rudi73}.

To define $H^{-k}_{0,\loc}(\R_{\geq0};X)$ for $k>0$, let $0<T_1<T_2$. Extension by zero on $[T_1,T_2]$ defines a continuous embedding of $H_{0r}^k([0,T_1];X)$ into $H_{0r}^k([0,T_2];X)$. Its dual defines the restriction operator from $H_{0l}^{-k}([0,T_2];X)$ to $H_{0l}^{-k}([0,T_1];X)$. An element $v\in H_{0,\loc}^{-k}(\R_{\geq0};X)$ is a family $(v_T)_{T>0}$ with $v_T\in H^{-k}_{0l}([0,T];X)$ such that, whenever $0<T_1<T_2$, the restriction of $v_{T_2}$ to $[0,T_1]$ equals $v_{T_1}$. We equip this space with the locally convex topology generated by the seminorms \[ v\mapsto\|v_T\|_{H^{-k}_{0l}([0,T];X)}, \qquad T>0. \] The support of $v=(v_T)_{T>0}$ is defined by \[ \supp v \coloneqq  \bigcup_{T>0} \big(\supp v_T\cap[0,T)\big). \] With these definitions, $H_{0,\loc}^{0}(\R_{\geq0};X)$ is identified with $L^2_{\loc}(\R_{\geq0};X)$.

The left derivative $\left(\frac{{\rm d}}{{\rm d}t}\right)_l$ extends naturally to $H_{0,\loc}^{k}(\R_{\ge0};X)$ for every $k\in\Z$. It defines a continuous bijection from $H_{0,\loc}^{k}(\R_{\ge0};X)$ onto $H_{0,\loc}^{k-1}(\R_{\ge0};X)$ with continuous inverse. Similarly, the $\tau$-right shift defines a continuous operator on $H_{0,\loc}^{k}(\R_{\ge0};X)$ for every $k\in\Z$, as follows from its boundedness on finite time intervals.

We conclude these preliminaries by introducing convolution. For $v,w\in L^{2}_{\loc}(\R_{\ge0};X)$, we define their convolution $v\ast w\in L^{2}_{\loc}(\R_{\ge0})$ by
\[
(v\ast w)(t)=\int_0^t \langle v(\tau),w(t-\tau)\rangle_X\,{\rm d}\tau.
\]
The convolution $v\ast w$ is continuous. Moreover, convolution is bilinear and commutative and defines a continuous mapping from
$L^{2}_{\loc}(\R_{\ge0};X)\times L^{2}_{\loc}(\R_{\ge0};X)$
to $L^{2}_{\loc}(\R_{\ge0})$. Its restriction to $[0,T]$ depends only on the restrictions of $v$ and $w$ to $[0,T]$.

If $v\in H^{1}_{0,\loc}(\R_{\ge0};X)$ and $w\in L^{2}_{\loc}(\R_{\ge0};X)$, then
$v\ast w\in H^{1}_{0,\loc}(\R_{\ge0})$ and
\[
\big(\tfrac{{\rm d}}{{\rm d}t}\big)_l(v\ast w)
=
\left(\big(\tfrac{{\rm d}}{{\rm d}t}\big)_lv\right)\ast w.
\]
Since the left derivative is a linear homeomorphism from $H^{k}_{0,\loc}(\R_{\geq0};X)$ onto $H^{k-1}_{0,\loc}(\R_{\geq0};X)$ for every $k\in\Z$, convolution extends uniquely to a continuous bilinear mapping 
\[ \ast: H^{k_1}_{0,\loc}(\R_{\geq0};X) \times H^{k_2}_{0,\loc}(\R_{\geq0};X) \to H^{k_1+k_2}_{0,\loc}(\R_{\geq0}) \] defined by 
\begin{equation} \begin{aligned} \big(\tfrac{{\rm d}}{{\rm d}t}\big)_l^{k_1+k_2}(v\ast w) &= \left(\big(\tfrac{{\rm d}}{{\rm d}t}\big)_l^{k_1}v\right) \ast \left(\big(\tfrac{{\rm d}}{{\rm d}t}\big)_l^{k_2}w\right) \end{aligned} \label{eq:convrule} \end{equation}
for all $k_1,k_2\in\Z$, $v\in H^{k_1}_{0,\loc}(\R_{\geq0};X)$, and $w\in H^{k_2}_{0,\loc}(\R_{\geq0};X)$. Here, negative powers of the left derivative denote powers of its continuous inverse.
The convolution $\varphi\ast v\in H^{k_1+k_2}_{0,\loc}(\R_{\geq0};X)$ of a scalar element $\varphi\in H^{k_1}_{0,\loc}(\R_{\geq0})$ and an $X$-valued element $v\in H^{k_2}_{0,\loc}(\R_{\geq0};X)$ is defined analogously.

\section{System nodes and trajectories}\label{sec:sysnode}

Let $X$, $U$, and $Y$ be Hilbert spaces, and let
\(A\&B:\dom(A\&B)\subset X\times U\to X\) and
\(C\&D:\dom(C\&D)\subset X\times U\to Y\)
be linear operators. In this section, we recall the basic concepts required for systems of the form \eqref{eq:ODEnode}.
In contrast to the finite-dimensional case, the operators $A\&B$ and $C\&D$ need not decompose into separate components associated with the state and the input. This feature is essential for the treatment of boundary control systems governed by partial differential equations.
The autonomous dynamics, corresponding to the trivial input $u\equiv0$, are determined by the so-called \emph{main operator}
$A\colon\dom(A)\subset X\to X$,
where $\dom(A)\coloneqq\setdef{x\in X}{\spvek{x}{0}\in\dom(A\&B)}$ and
$Ax\coloneqq A\&B\spvek{x}{0}$ for all $x\in\dom(A)$.

\begin{defn}[System node]\label{def:sysnode}
A {\em system node} on the triple $(X,U,Y)$ of Hilbert spaces is a linear operator
$S=\sbvek{A\&B}{C\&D}$, where
$A\&B:\dom(A\&B)\subset X\times U\to X$ and
$C\&D:\dom(C\&D)\subset X\times U\to Y$, satisfying the following conditions:
\begin{enumerate}[(a)]
    \item\label{def:sysnodea} $A\&B$ is closed.
    \item\label{def:sysnodeb} $C\&D\in L(\dom(A\&B),Y)$.
    \item\label{def:sysnodec} For every $u\in U$, there exists some $x\in X$ such that $\spvek{x}{u}\in\dom(S)$.
    \item\label{def:sysnoded} The main operator $A$ generates a strongly continuous semigroup
    $\mathfrak{A}(\cdot)\colon\R_{\geq0}\to L(X)$ on $X$.
\end{enumerate}
\end{defn}
These properties imply that $S$ is closed and that $\dom(S)=\dom(A\&B)$.

We next introduce classical and generalised trajectories of \eqref{eq:ODEnode}.
The space $C(\R_{\geq0};X)$ is equipped with the locally convex topology generated by the seminorms given by the maximum norm on intervals $[0,T]$, $T>0$.

\begin{defn}[Classical/generalised trajectories]\label{def:traj}
Let $S=\sbvek{A\&B}{C\&D}$ be a system node on $(X,U,Y)$.
A {\em classical trajectory} of \eqref{eq:ODEnode} is a triple
\[
(x,u,y)\in C^1(\R_{\geq0};X)\times C(\R_{\geq0};U)\times C(\R_{\geq0};Y)
\]
that satisfies \eqref{eq:ODEnode} for every $t>0$.
A {\em generalised trajectory} of \eqref{eq:ODEnode} is a limit of classical trajectories in the topology of $C(\R_{\geq0};X)\times L^2_{\loc}(\R_{\geq0};U)\times L^2_{\loc}(\R_{\geq0};Y)$.
\end{defn}

Any operator $A\&B$ satisfying conditions \eqref{def:sysnodea}, \eqref{def:sysnodec}, and \eqref{def:sysnoded} of Definition~\ref{def:sysnode} can be regarded as a system node on $(X,U,\{0\})$. Accordingly, we also speak of classical and generalised trajectories $(x,u)$ of
\begin{align}\label{eq:ODE}
\dot{x}=A\&B\spvek{x}{u}.
\end{align}

We next recall a solvability result from \cite{Staffans2005}. In addition to sufficient regularity of the input, it requires compatibility between the initial state and the initial input value in the sense that the corresponding pair belongs to $\dom(A\&B)$. For boundary control systems, this means that the boundary value of the prescribed initial state is consistent with the input at $t=0$.

\begin{prop}[Existence of classical trajectories {\cite[Lem.~4.7.8]{Staffans2005}}]\label{prop:solex}
Let $S$ be a system node on $(X,U,Y)$, let $x_0\in X$, and let
$u\in W^{2,1}_{\loc}(\R_{\geq0};U)$ satisfy
$\spvek{x_0}{u(0)}\in\dom S$. Then there exist unique functions
$x\in C^1(\R_{\geq0};X)$ and $y\in C(\R_{\geq0};Y)$ such that
$x(0)=x_0$ and $(x,u,y)$ is a classical trajectory of \eqref{eq:ODEnode}.
\end{prop}

To formulate further results on the existence and regularity of trajectories, it is useful to separate $A\&B$ into operators associated with the state and the input, as is customary in the theory of infinite-dimensional systems; see, for instance, \cite{TuWe09}. This separation requires the introduction of extrapolation spaces containing $X$ as a continuously and densely embedded subspace.

\begin{rem}[System nodes]\label{rem:nodes}
Let $S=\sbvek{A\&B}{C\&D}$ be a system node on $(X,U,Y)$.
\begin{enumerate}[(a)]
\item\label{rem:nodesa}
For $k\in\N$, let $X_k\coloneqq \dom(A^k)$, and let $X_{-k}$ be the completion of $X$ with respect to the norm
$\|x\|_{X_{-k}}\coloneqq \|(\lambda\Id-A)^{-k}x\|$, where
$\lambda\in\R$ belongs to the resolvent set of $A$.
For $k\in\Z$, the operator $A$ restricts or extends, as appropriate, to a closed and densely defined operator
$A_k:X_k\supset\dom A_k=X_{k+1}\to X_k$.
Likewise, the semigroup $\frakA(\cdot)$ generated by $A$ restricts or extends to a semigroup $\frakA_k(\cdot)$ on $X_k$, whose generator is $A_k$; see \cite[Prop.~2.10.3 \& 2.10.4]{TuWe09}.

\item\label{rem:nodesb}
There exists an operator $B\in L(U,X_{-1})$ such that
$[A_{-1}\ B]\in L(X\times U,X_{-1})$ extends $A\&B$. Moreover,
\[
\dom(A\&B)=
\setdef{\spvek{x}{u}\in X\times U}{A_{-1}x+Bu\in X};
\]
see \cite[Def.~4.7.2 \& Lem.~4.7.3]{Staffans2005}.

\item\label{rem:nodesc}
For $k\in\Z$, let $X_{d,k}$ denote the scale of spaces constructed as in part~\eqref{rem:nodesa}, but starting from $A^*$.
Then \cite[Prop.~2.10.2]{TuWe09} yields
$X_{d,k}=X_{-k}^\prime$, where the latter denotes the dual of $X_{-k}$ with respect to the pivot space $X$.
\end{enumerate}
\end{rem}

We collect some further properties of classical and generalised trajectories.

\begin{rem}[Classical/generalised trajectories]\label{rem:sols}
Let $S=\sbvek{A\&B}{C\&D}$ be a system node on $(X,U,Y)$.
\begin{enumerate}[(a)]
\item\label{rem:sols1}
If $(x,u)$ is a classical trajectory of
$\dot{x}=A\&B\spvek{x}{u}$, then
\[
\spvek{x}{u}\in C(\R_{\geq0};\dom S).
\]

\item\label{rem:sols2}
A pair $(x,u)\in
C(\R_{\geq0};X)
\times L^2_{\loc}(\R_{\geq0};U)$
is a generalised trajectory of
$\dot{x}=A\&B\spvek{x}{u}$ if and only if
\begin{equation}
\forall\,t\geq0:\quad
x(t)=\frakA(t)x(0)
+\int_0^t\frakA_{-1}(t-\tau)Bu(\tau)\,{\rm d}\tau,
\label{eq:mildsol}
\end{equation}
where the integral is understood as an $X_{-1}$-valued integral.

The output expression
\[
y(t)=C\&D\spvek{x(t)}{u(t)}
\]
is not necessarily defined pointwise. However,
\cite[Lem.~4.7.9]{Staffans2005} shows that the second time primitive
\[
g
\coloneqq 
\int_0^\cdot(\cdot-\tau)
\spvek{x(\tau)}{u(\tau)}\,{\rm d}\tau
\]
belongs to
\[
C(\R_{\geq0};\dom(A\&B))
\subset
L^2_{\loc}(\R_{\geq0};\dom(A\&B)).
\]
We may therefore define the distributional output by
\begin{equation}
y
\coloneqq 
\big(\tfrac{{\rm d}}{{\rm d}t}\big)_l^2
\big(C\&Dg\big)
=
C\&D\spvek{x}{u}
\in H^{-2}_{0,\loc}(\R_{\geq0};Y),
\label{eq:ydef}
\end{equation}
where the final equality follows from the pointwise application of
bounded operators introduced in \eqref{eq:Apointwl}.

Consequently, a triple
\[
(x,u,y)\in
C(\R_{\geq0};X)
\times L^2_{\loc}(\R_{\geq0};U)
\times L^2_{\loc}(\R_{\geq0};Y)
\]
is a generalised trajectory of \eqref{eq:ODEnode} if and only if
$(x,u)$ is a generalised trajectory of \eqref{eq:ODE} and the
distribution defined in \eqref{eq:ydef} is represented by the
$L^2_{\loc}$-function $y$.

\item
Consider the autonomous case $u=0$, and define
$C\in L(\dom(A),Y)$ by
\[
Cx=C\&D\spvek{x}{0}.
\]
If $x_0\in X_1=\dom(A)$, then
Remark~\ref{rem:nodes}\,\eqref{rem:nodesa} yields
\[
x=\frakA(\cdot)x_0\in C(\R_{\geq0};\dom A),
\]
and hence
\[
y\coloneqq Cx\in C(\R_{\geq0};Y)
\subset L^2_{\loc}(\R_{\geq0};Y).
\]
To define the output for less regular initial states, fix
$\lambda\in\rho(A)$. For every $j\in\Z$, the operator
\[
\lambda\Id-A_j:X_{j+1}\to X_j
\]
is an isomorphism and commutes with the corresponding semigroups.
Moreover,
\[
\frakA_j(\cdot)(\lambda\Id-A_j)x_0
=
\left(
\lambda\Id-\tfrac{{\rm d}}{{\rm d}t}
\right)
\frakA_{j+1}(\cdot)x_0,
\qquad
x_0\in X_{j+1},
\]
where $\tfrac{{\rm d}}{{\rm d}t}$ denotes the distributional
derivative; see \cite[Chap.~II, Lem.~1.3]{EngeNage00}.
By iteration, it follows that
\[
\frakA(\cdot)x_0
\in
\begin{cases}
H^{k-1}_{\loc}(\R_{\geq0};X_1),
& k\geq1,\\
H^{k-1}_{0,\loc}(\R_{\geq0};X_1),
& k\leq0,
\end{cases}
\qquad
x_0\in X_k.
\]
For $k\in\Z$, let $\mathcal{Y}_k$ denote
$H^{k-1}_{\loc}(\R_{\geq0};Y)$ if $k\geq1$ and
$H^{k-1}_{0,\loc}(\R_{\geq0};Y)$ if $k\leq0$.
Since $C\in L(X_1,Y)$, its pointwise or distributional application,
as introduced in \eqref{eq:Apointwl}, gives
\[
\forall\,k\in\Z,\ x_0\in X_k:\quad
C\frakA(\cdot)x_0
\in\mathcal{Y}_k.
\]
Thus, the {\em state-to-output map}
\begin{equation}
\begin{aligned}
\frakC_k:\quad
X_k&\to\mathcal{Y}_k,\\
x_0&\mapsto C\frakA(\cdot)x_0,
\end{aligned}
\label{eq:genSOmap}
\end{equation}
is well defined and continuous.

\end{enumerate}
\end{rem}

Remark~\ref{rem:sols} motivates the introduction of the continuous
{\em input-to-state map}
\begin{align}
\frakS:\quad
L^2_{\loc}(\R_{\geq0};U)
&\to L^2_{\loc}(\R_{\geq0};X_{-1}),
\label{eq:genISmap}\\
u&\mapsto
\int_0^\cdot
\frakA_{-1}(\cdot-\tau)Bu(\tau)\,{\rm d}\tau.
\nonumber
\end{align}
It maps an input to the corresponding state trajectory with zero initial
state. For every $T>0$, the restriction of $\frakS u$ to $[0,T]$ belongs
to $C([0,T];X_{-1})$. We nevertheless use
$L^2_{\loc}(\R_{\geq0};X_{-1})$ as the codomain in
\eqref{eq:genISmap}, since this allows us to incorporate $\frakS$ into
the Sobolev-space framework introduced above.

Remark~\ref{rem:sols}\,\eqref{rem:sols2} further implies that the mapping
\begin{equation}
u\mapsto\spvek{\frakS u}{u}
\label{eq:uA&B}
\end{equation}
is continuous from $L^2_{\loc}(\R_{\geq0};U)$ to
$H^{-2}_{0,\loc}(\R_{\geq0};\dom(A\&B))$.

In this notation, well-posedness is equivalent to the following three
properties: the map $\frakC_0$ takes values in
$L^2_{\loc}(\R_{\geq0};Y)$, the map $\frakS$ takes values in
$C(\R_{\geq0};X)$, and
\[
u\mapsto C\&D\spvek{\frakS u}{u}
\]
maps $L^2_{\loc}(\R_{\geq0};U)$ into
$L^2_{\loc}(\R_{\geq0};Y)$.

It follows from Remark~\ref{rem:sols} that
\[
(x,u,y)\in
C(\R_{\geq0};X)
\times L^2_{\loc}(\R_{\geq0};U)
\times L^2_{\loc}(\R_{\geq0};Y)
\]
is a generalised trajectory of \eqref{eq:ODEnode} if and only if there
exists some $x_0\in X$ such that
\[
x=\frakA(\cdot)x_0+\frakS u,
\qquad
y=\frakC_0x_0+C\&D\spvek{\frakS u}{u}.
\]
This characterisation provides the basis for the distributional
extension of the trajectory concept.

By shift invariance,
\[
\forall\,u\in C_0^\infty(\R_{>0};U):\quad
\frakS\big(\tfrac{{\rm d}}{{\rm d}t}\big)_lu
=
\big(\tfrac{{\rm d}}{{\rm d}t}\big)_l\frakS u.
\]
Since $C_0^\infty(\R_{>0};U)$ is dense in
$H^1_{0,\loc}(\R_{\geq0};U)$, this identity extends by continuity to
every $u\in H^1_{0,\loc}(\R_{\geq0};U)$. In particular,
\[
\frakS u\in H^1_{0,\loc}(\R_{\geq0};X_{-1})
\]
and
\[
\frakS\big(\tfrac{{\rm d}}{{\rm d}t}\big)_lu
=
\big(\tfrac{{\rm d}}{{\rm d}t}\big)_l\frakS u.
\]
Iteration shows that $\frakS$ restricts to a continuous operator from
$H^k_{0,\loc}(\R_{\geq0};U)$ to
$H^k_{0,\loc}(\R_{\geq0};X_{-1})$ for every $k\in\N_0$.

For $m\in\N$, the continuous extension to
$H^{-m}_{0,\loc}(\R_{\geq0};U)$ is uniquely determined by
$\frakS_{-m}u
\coloneqq 
\big(\tfrac{{\rm d}}{{\rm d}t}\big)_l^m
\frakS
\big(\tfrac{{\rm d}}{{\rm d}t}\big)_l^{-m}u$.
Thus, for every $k\in\Z$, we obtain the continuous
{\em distributional input-to-state map}
\begin{align}
\frakS_k:\quad
H^k_{0,\loc}(\R_{\geq0};U)
&\to H^k_{0,\loc}(\R_{\geq0};X_{-1}),
\label{eq:DistrOp1}
\end{align}
where $\frakS_0=\frakS$.

Applying the same extension procedure to the continuous mapping
\eqref{eq:uA&B} yields
\begin{equation}
\sbvek{\frakS_k}{\Id}:
H^k_{0,\loc}(\R_{\geq0};U)
\to
H^{k-2}_{0,\loc}(\R_{\geq0};\dom(A\&B))
\quad\text{is continuous}.
\label{eq:DistrOp2}
\end{equation}

The distributional input-to-state map commutes with scalar convolution.

\begin{lem}\label{lem:convsol}
Let $S$ be a system node on $(X,U,Y)$, let $k,l\in\Z$, and let
\[
u\in H^k_{0,\loc}(\R_{\geq0};U),
\qquad
\alpha\in H^l_{0,\loc}(\R_{\geq0}).
\]
Then
\[
\frakS_{k+l}(\alpha\ast u)
=
\alpha\ast\frakS_ku.
\]
\end{lem}

\begin{proof}
For $\alpha\in C_0^\infty(\R_{>0})$ and
$u\in C_0^\infty(\R_{>0};U)$, the assertion follows from
\eqref{eq:genISmap} by changing the order of integration. The general
case follows from the density of $C_0^\infty(\R_{>0})$ and
$C_0^\infty(\R_{>0};U)$ in the corresponding spaces
$H^l_{0,\loc}(\R_{\geq0})$ and
$H^k_{0,\loc}(\R_{\geq0};U)$, respectively, together with the
continuity of convolution and of the operators $\frakS_j$.
\end{proof}

\section{Duality and the modulating function}
\label{sec:dual}

In this section, we develop the modulating-function framework for system nodes. We first introduce the adjoint system node and then formulate a distributional null-control problem that forms the basis of our reconstruction result.

Let $S=\sbvek{A\&B}{C\&D}$ be a system node on $(X,U,Y)$. Its {\em adjoint system node}
\[
S^*=\left(\sbvek{A\&B}{C\&D}\right)^*
\]
is again a system node by \cite[Lem.~6.2.14]{Staffans2005}. We write
\begin{equation}\label{eq:Sstar}
S^*=\sbvek{A^d\&B^d}{C^d\&D^d}
\end{equation}
and consider the corresponding {\em adjoint system}
\begin{equation}
\spvek{\dot{\varphi}(t)}{\mu(t)}
=
\sbvek{A^d\&B^d}{C^d\&D^d}
\spvek{\varphi(t)}{\eta(t)}.
\label{eq:ODEnodeadj}
\end{equation}
Moreover, \cite[Lem.~6.2.14]{Staffans2005} shows that the main operator of $S^*$ is the adjoint of the main operator of $S$, that is, $A^d=A^*$.

We now formulate the main idea underlying the modulating-function approach. Our framework extends the classical theory in two respects: it applies to infinite-dimensional systems represented by system nodes, and it allows for distributional null controls.

Let $K\&L\in L(\dom(A\&B),\R)=\dom(A\&B)^\prime$. For a system \eqref{eq:ODEnode}, we aim to reconstruct the additional output
\begin{equation}
z(t)=K\&L\spvek{x(t)}{u(t)}
\label{eq:zout}
\end{equation}
from the past input $u$ and output $y$. It follows directly from the definition of a system node that
\begin{equation}
S_{\rm ext}
=
\left[
\begin{smallmatrix}
A\&B\\
C\&D\\
K\&L
\end{smallmatrix}
\right]
\label{eq:Sext}
\end{equation}
is a system node on $(X,U,Y\times\R)$.

Writing the adjoint of $S_{\rm ext}$ as
\begin{equation}
S_{\rm ext}^*
=
\left[
\begin{smallmatrix}
A^d\&[B^d\,G]\\
C^d\&[D^d\,H]
\end{smallmatrix}
\right],
\label{eq:Sextadj}
\end{equation}
we consider a distributional input $\eta$ such that
\begin{equation}
\spvek{\dot{\varphi}}{\mu}
=
\left[
\begin{smallmatrix}
A^d\&[B^d\,G]\\
C^d\&[D^d\,H]
\end{smallmatrix}
\right]
\left(
\begin{smallmatrix}
\varphi\\
\eta\\
\delta
\end{smallmatrix}
\right),
\label{eq:deltanullcontrol}
\end{equation}
where $\delta\in H^{-1}_{0,\loc}(\R_{\geq0})$ denotes the Dirac distribution at zero. We require $\varphi$, $\eta$, and $\mu$ to vanish on $[T,\infty)$ for some $T>0$.

More precisely, let $k\in\N_0$, let
$\eta\in H^{-k}_{0,\loc}(\R_{\geq0};Y)$, and set
$r\coloneqq \max\{k,1\}$.
Then
\[
\spvek{\eta}{\delta}
\in
H^{-r}_{0,\loc}(\R_{\geq0};Y\times\R).
\]
Let $\frakS^d_{-r,{\rm ext}}$ denote the distributional
input-to-state map associated with the system node $S_{\rm ext}^*$.
We require that there exist
\[
\varphi\in H^{-k}_{0,\loc}(\R_{\geq0};X),
\qquad
\mu\in H^{-k-2}_{0,\loc}(\R_{\geq0};U),
\]
such that
\begin{equation}
\spvek{\varphi}{\mu}
=
\left[
\begin{smallmatrix}
\Id\qquad0\quad\\
C^d\&[D^d\,H]
\end{smallmatrix}
\right]
\sbvek{\frakS^d_{-r,{\rm ext}}}{\Id}
\spvek{\eta}{\delta},
\label{eq:extadjcontrol}
\end{equation}
and
\begin{equation}
\supp\varphi\subset[0,T],\qquad
\supp\eta\subset[0,T],\qquad
\supp\mu\subset[0,T].
\label{eq:extadjsupp}
\end{equation}
An input $\eta$ satisfying \eqref{eq:extadjcontrol} and
\eqref{eq:extadjsupp} is called a {\em generalised null control in
time $T$}.

This terminology is consistent with the classical notion of null control. If $\varphi$, $\eta$, and $\mu$ are functions rather than genuine distributions, then \eqref{eq:extadjsupp} means that the state $\varphi$ is driven to zero by time $T$.

A particularly important special case arises when the additional scalar output depends only on the state. Suppose that there exists some $\varphi_0\in X$ such that
\[
K\&L\spvek{x}{u}
=
\langle x,\varphi_0\rangle_X
\qquad
\forall\,\spvek{x}{u}\in\dom(A\&B).
\]
Then $\dom(S_{\rm ext}^*)=\dom(S^*)\times\R$
and
\[
S_{\rm ext}^*
\left(
\begin{smallmatrix}
\varphi\\
\eta\\
\lambda
\end{smallmatrix}
\right)
=
\sbvek{A^d\&B^d}{C^d\&D^d}
\spvek{\varphi}{\eta}
+
\spvek{\lambda\varphi_0}{0}
\]
for all $\spvek{\varphi}{\eta}\in\dom(S^*)$ and $\lambda\in\R$. In this case, \eqref{eq:deltanullcontrol} describes the adjoint system with the initial state $\varphi_0$, which is driven to zero in time $T$.

In the distributional setting, the moving-horizon integral in \eqref{eq:movhor_inf} is replaced by the convolution expression $u\ast\mu-y\ast\eta$.
We say that this convolution coincides with the additional output
$z=K\&L\spvek{x}{u}$ after time $T$ if
\begin{equation}
\supp\big(u\ast\mu-y\ast\eta-z\big)\subset[0,T].
\label{eq:maineq}
\end{equation}
Thus, the two distributions agree on $(T,\infty)$.

We are now in a position to state the main result of this section. It shows that a generalised null control of the extended adjoint system yields a modulating distribution that reconstructs the additional output $z$. This distributional formulation substantially broadens the scope of the method and is particularly relevant for the realisation of unbounded state feedback considered in Section~\ref{sec:feedback}.
\begin{thm}\label{thm:main}
Let $S=\sbvek{A\&B\\[-2mm]}{C\&D}$ be a system node on $(X,U,Y)$, let
\[
K\&L\in L(\dom(A\&B),\R)=\dom(A\&B)^\prime,
\]
and let $(x,u,y)$ be a generalised trajectory of \eqref{eq:ODEnode}.
Let $S_{\rm ext}$ be defined as in \eqref{eq:Sext}.
Let $T>0$ and assume that, for some $k\in\N_0$,
\[
\eta\in H^{-k}_{0,\loc}(\R_{\geq0};Y),\qquad
\varphi\in H^{-k}_{0,\loc}(\R_{\geq0};X),\qquad
\mu\in H^{-k-2}_{0,\loc}(\R_{\geq0};U),
\]
and that \eqref{eq:extadjcontrol} and \eqref{eq:extadjsupp} hold.
Then, for
\[
z=K\&L\spvek{x}{u}\in H^{-2}_{0,\loc}(\R_{\geq0})
\subset H^{-k-2}_{0,\loc}(\R_{\geq0}),
\]
the identity \eqref{eq:maineq} holds.
\end{thm}

\begin{proof}
We first prove the assertion under the additional assumption that
$(x,u,y)$ is a classical trajectory of \eqref{eq:ODEnode}.

Let $(\alpha_n)$ be a mollifier sequence in
$C_0^\infty(\R_{>0})$, that is, for every $n\in\N$,
\begin{enumerate}[(i)]
    \item $\int_0^\infty\alpha_n(t)\,{\rm d}t=1$,
    \item $\supp\alpha_n\subset(0,1/n]$, and
    \item $\alpha_n(t)\geq0$ for all $t\in\R_{\geq0}$.
\end{enumerate}
Define $\varphi_n\coloneqq \alpha_n\ast\varphi$, $\eta_n\coloneqq \alpha_n\ast\eta$, $\mu_n\coloneqq \alpha_n\ast\mu$.
Since $\alpha_n\ast\delta=\alpha_n$, Lemma~\ref{lem:convsol},
\eqref{eq:extadjcontrol}, and \eqref{eq:DistrOp2}, applied to the
system node $S_{\rm ext}^*$, imply that
\[
\left(
\begin{smallmatrix}
\varphi_n\\
\eta_n\\
\alpha_n
\end{smallmatrix}
\right)
=
\alpha_n\ast
\left(
\begin{smallmatrix}
\varphi\\
\eta\\
\delta
\end{smallmatrix}
\right)
\in
C^\infty(\R_{\geq0};\dom(S_{\rm ext}^*)).
\]
In particular, the adjoint system equation holds pointwise:
\[
\forall\,t\geq0:\quad
\spvek{\dot{\varphi}_n(t)}{\mu_n(t)}
=
S_{\rm ext}^*
\left(
\begin{smallmatrix}
\varphi_n(t)\\
\eta_n(t)\\
\alpha_n(t)
\end{smallmatrix}
\right)
=
\left[
\begin{smallmatrix}
A^d\&[B^d\,G]\\
C^d\&[D^d\,H]
\end{smallmatrix}
\right]
\left(
\begin{smallmatrix}
\varphi_n(t)\\
\eta_n(t)\\
\alpha_n(t)
\end{smallmatrix}
\right).
\]
Moreover, $\varphi_n$ vanishes in a neighbourhood of zero and hence
$\varphi_n(0)=0$.

Using the support property of convolution, together with
\eqref{eq:extadjsupp} and
$\supp\alpha_n\subset(0,1/n]$, gives
\begin{equation}
\supp\varphi_n\subset[0,T+1/n],\qquad
\supp\eta_n\subset[0,T+1/n],\qquad
\supp\mu_n\subset[0,T+1/n].
\label{eq:extadjsupp_ext}
\end{equation}

For $t\geq0$ and $\tau\in[0,t]$, the product rule yields
\begin{align*}
&\phantom{=}
\tfrac{\mathrm{d}}{\mathrm{d}\tau}
\langle x(t-\tau),\varphi_n(\tau)\rangle_X
\\
&=
\langle x(t-\tau),\dot{\varphi}_n(\tau)\rangle_X
-
\langle\dot{x}(t-\tau),\varphi_n(\tau)\rangle_X
\\
&=
\left\langle
x(t-\tau),
A^d\&[B^d\,G]
\left(
\begin{smallmatrix}
\varphi_n(\tau)\\
\eta_n(\tau)\\
\alpha_n(\tau)
\end{smallmatrix}
\right)
\right\rangle_X
-
\left\langle
A\&B\spvek{x(t-\tau)}{u(t-\tau)},
\varphi_n(\tau)
\right\rangle_X
\\
&=
\underbrace{
\left\langle
\spvek{x(t-\tau)}{u(t-\tau)},
S_{\rm ext}^*
\left(
\begin{smallmatrix}
\varphi_n(\tau)\\
\eta_n(\tau)\\
\alpha_n(\tau)
\end{smallmatrix}
\right)
\right\rangle_{X\times U}
-
\left\langle
S_{\rm ext}\spvek{x(t-\tau)}{u(t-\tau)},
\left(
\begin{smallmatrix}
\varphi_n(\tau)\\
\eta_n(\tau)\\
\alpha_n(\tau)
\end{smallmatrix}
\right)
\right\rangle_{X\times Y\times\R}
}_{=0}
\\
&\qquad
-\langle u(t-\tau),\mu_n(\tau)\rangle_U
+\langle y(t-\tau),\eta_n(\tau)\rangle_Y
+\alpha_n(\tau)z(t-\tau).
\end{align*}
Integrating with respect to $\tau$ over $[0,t]$ and using
$\varphi_n(0)=0$, we obtain
\[
\forall\,t\geq0:\quad
\langle x(0),\varphi_n(t)\rangle_X
=
-(u\ast\mu_n)(t)
+(y\ast\eta_n)(t)
+(\alpha_n\ast z)(t).
\]
Since \eqref{eq:extadjsupp_ext} implies that
$\varphi_n(t)=0$ for all $t\geq T+1/n$, it follows that
\begin{equation}
\supp\left(
u\ast\mu_n-y\ast\eta_n-\alpha_n\ast z
\right)
\subset[0,T+1/n]
\qquad
\forall\,n\in\N.
\label{eq:conv_n}
\end{equation}
By \cite[Thm.~4.15]{Alt16}, for every Hilbert space $E$ and every
$w\in L^2_{\loc}(\R_{\geq0};E)$,
\[
\alpha_n\ast w\longrightarrow w
\quad\text{in }L^2_{\loc}(\R_{\geq0};E).
\]
This approximation property extends to all orders of the Sobolev scale.
Indeed, let $j\in\Z$ and
$w\in H^j_{0,\loc}(\R_{\geq0};E)$. By \eqref{eq:convrule},
\[
\big(\tfrac{{\rm d}}{{\rm d}t}\big)_l^j
(\alpha_n\ast w)
=
\alpha_n\ast
\big(\tfrac{{\rm d}}{{\rm d}t}\big)_l^j w.
\]
Since
$\big(\tfrac{{\rm d}}{{\rm d}t}\big)_l^j w
\in L^2_{\loc}(\R_{\geq0};E)$
and
$\big(\tfrac{{\rm d}}{{\rm d}t}\big)_l^j$
is a homeomorphism from
$H^j_{0,\loc}(\R_{\geq0};E)$ onto the space
$L^2_{\loc}(\R_{\geq0};E)$, it follows that
\begin{equation}
\alpha_n\ast w\longrightarrow w
\quad\text{in }H^j_{0,\loc}(\R_{\geq0};E).
\label{eq:mollconv}
\end{equation}
In particular,
\[
\mu_n\longrightarrow\mu
\;\text{in }H^{-k-2}_{0,\loc}(\R_{\geq0};U),
\quad
\eta_n\longrightarrow\eta
\;\text{in }H^{-k}_{0,\loc}(\R_{\geq0};Y),
\quad
\alpha_n\ast z\longrightarrow z
\;\text{in }H^{-2}_{0,\loc}(\R_{\geq0}).
\]
By continuity of convolution,
\[
u\ast\mu_n\longrightarrow u\ast\mu
\;\text{in }H^{-k-2}_{0,\loc}(\R_{\geq0})
\quad\text{and}
\quad
y\ast\eta_n\longrightarrow y\ast\eta
\;\text{in }H^{-k}_{0,\loc}(\R_{\geq0}).
\]
Using the continuous embeddings
\[
H^{-k}_{0,\loc}(\R_{\geq0})
\subset
H^{-k-2}_{0,\loc}(\R_{\geq0}),
\qquad
H^{-2}_{0,\loc}(\R_{\geq0})
\subset
H^{-k-2}_{0,\loc}(\R_{\geq0}),
\]
we conclude that
\[
u\ast\mu_n-y\ast\eta_n-\alpha_n\ast z
\longrightarrow
u\ast\mu-y\ast\eta-z\;\text{in }H^{-k-2}_{0,\loc}(\R_{\geq0}).
\]
Let $\psi$ be a test function with compact support in $(T,\infty)$.
For all sufficiently large $n$, its support is disjoint from
$[0,T+1/n]$. Hence, by \eqref{eq:conv_n},
\[
\left\langle
u\ast\mu_n-y\ast\eta_n-\alpha_n\ast z,\psi
\right\rangle
=0.
\]
Passing to the limit gives
\[
\left\langle
u\ast\mu-y\ast\eta-z,\psi
\right\rangle
=0.
\]
Therefore,
\[
\supp\big(u\ast\mu-y\ast\eta-z\big)\subset[0,T],
\]
which proves the assertion for classical trajectories.

Now let $(x,u,y)$ be a generalised trajectory of
\eqref{eq:ODEnode}. By Definition~\ref{def:traj}, there exists a
sequence of classical trajectories $(x_j,u_j,y_j)$ such that
\[
x_j\to x
\quad\text{in }C(\R_{\geq0};X),
\qquad
u_j\to u
\quad\text{in }L^2_{\loc}(\R_{\geq0};U),
\qquad
y_j\to y
\quad\text{in }L^2_{\loc}(\R_{\geq0};Y).
\]
Set $z_j\coloneqq K\&L\spvek{x_j}{u_j}$.
We claim that
$z_j\to z\coloneqq K\&L\spvek{x}{u}$ in $H^{-2}_{0,\loc}(\R_{\geq0})$.
Indeed, let
\[
g_j
\coloneqq 
\int_0^\cdot(\cdot-\tau)
\spvek{x_j(\tau)-x(\tau)}{u_j(\tau)-u(\tau)}
\,{\rm d}\tau.
\]
Then $g_j$ converges to zero locally in $X\times U$. Moreover, the
distributional state equation gives
\[
A\&B\,g_j(t)
=
\int_0^t\big(x_j(\tau)-x(\tau)\big)\,{\rm d}\tau
-
t\big(x_j(0)-x(0)\big).
\]
Consequently, $g_j\to0$  in $L^2_{\loc}(\R_{\geq0};\dom(A\&B))$.
Since $K\&L\in L(\dom(A\&B),\R)$, it follows that
\[
z_j-z
=
\big(\tfrac{{\rm d}}{{\rm d}t}\big)_l^2(K\&L\,g_j)
\longrightarrow0
\;\text{in }H^{-2}_{0,\loc}(\R_{\geq0}).
\]
By continuity of convolution,
\[
u_j\ast\mu-y_j\ast\eta-z_j
\longrightarrow
u\ast\mu-y\ast\eta-z\;\text{in }H^{-k-2}_{0,\loc}(\R_{\geq0}).
\]
The result already proved for
classical trajectories gives
\[
\supp\big(u_j\ast\mu-y_j\ast\eta-z_j\big)\subset[0,T]
\qquad
\forall\,j\in\N.
\]
Testing against functions with compact support in $(T,\infty)$ and
passing to the limit therefore yields
\[
\supp\big(u\ast\mu-y\ast\eta-z\big)\subset[0,T].
\]
This proves \eqref{eq:maineq} for generalised trajectories.
\end{proof}

\section{State estimation}\label{sec:stateest}

Throughout this section, let $S$ be a system node on $(X,U,Y)$, where the Hilbert space $X$ is separable. Let $D\subset X$ be a dense subspace, and let $T>0$. We assume that every $\varphi_0\in D$ is distributionally null-controllable in time $T$ by the adjoint system \eqref{eq:ODEnodeadj}, in the following sense.

Let $\frakC^d_0:X\to H^{-1}_{0,\loc}(\R_{\geq0};U)$ denote the state-to-output map introduced in \eqref{eq:genSOmap}. For $k\in\N_0$, let $\frakS_{-k}^d$ denote the distributional input-to-state map of the adjoint system \eqref{eq:ODEnodeadj}. For $\varphi_0\in X$ and
$\eta\in H^{-k}_{0,\loc}(\R_{\geq0};Y)$, set
\[
\varphi
\coloneqq
\frakA^*(\cdot)\varphi_0+\frakS_{-k}^d\eta
\]
and let
\[
\mu
\coloneqq
\frakC^d_0\varphi_0
+
C^d\&D^d
\spvek{\frakS_{-k}^d\eta}{\eta}
\]
be the corresponding distributional output. The input $\eta$ is called
a {\em distributional null control for $\varphi_0\in X$ in time $T$}
if
\[
\supp\varphi\subset[0,T],
\qquad
\supp\eta\subset[0,T],
\qquad
\supp\mu\subset[0,T].
\]
If $\varphi$ is represented by a continuous function, as is the case
for the state of a generalised trajectory of
\eqref{eq:ODEnodeadj}, the first condition means that the initial state
$\varphi_0$ is driven to zero by time $T$. The above definition extends
this notion to the distributional setting.

Moreover, for $\varphi_0\in X$ and
$K\&L\in\dom(A\&B)^\prime$ satisfying
\[
K\&L\spvek{x}{u}
=
\langle x,\varphi_0\rangle_X
\qquad
\forall\,\spvek{x}{u}\in\dom(A\&B),
\]
the distributional null-control problem introduced above is equivalent to the generalised null-control problem \eqref{eq:extadjcontrol}, \eqref{eq:extadjsupp}.

Since $D$ is dense in the separable Hilbert space $X$, the Gram--Schmidt procedure yields an orthonormal basis
$(\varphi_{0j})_{j\in\N}$ of $X$ with
$\varphi_{0j}\in D$ for every $j\in\N$. For each $j\in\N$, let
\[
\eta_j\in H^{-k_j}_{0,\loc}(\R_{\geq0};Y),
\qquad k_j\in\N_0,
\]
be a distributional null control for $\varphi_{0j}$ in time $T$, and let
\[
\mu_j\in H^{-k_j-2}_{0,\loc}(\R_{\geq0};U)
\]
be the corresponding output.

Let $(x,u,y)$ be a generalised trajectory of \eqref{eq:ODEnode}. By Theorem~\ref{thm:main}, the distribution
\[
z_j\coloneqq u\ast\mu_j-y\ast\eta_j
\]
coincides on $(T,\infty)$ with the continuous function
\[
t\mapsto\langle x(t),\varphi_{0j}\rangle_X.
\]
Consequently, the restriction of $z_j$ to $(T,\infty)$ is represented by a continuous function. If, in addition, $\eta_j$ and $\mu_j$ belong to $L^2$, then \eqref{eq:movhor_inf} gives
\[
\forall\,t\geq T:\quad
\langle x(t),\varphi_{0j}\rangle_X
=
\int_0^T
\langle u(t-\tau),\mu_j(\tau)\rangle_U
-
\langle y(t-\tau),\eta_j(\tau)\rangle_Y
\,{\rm d}\tau.
\]

Thus, the Fourier coefficients
$\langle x(t),\varphi_{0j}\rangle_X$
of the state can be reconstructed from the input and output of
\eqref{eq:ODEnode} using precomputed null controls of the adjoint system. For $N\in\N$, define
\[
w_N
\coloneqq 
\sum_{j=1}^N
\big(u\ast\mu_j-y\ast\eta_j\big)\varphi_{0j}.
\]
On $(T,\infty)$, this distribution is represented by the continuous $X$-valued function
\[
t\mapsto
\sum_{j=1}^N
\langle x(t),\varphi_{0j}\rangle_X\varphi_{0j},
\]
which is the orthogonal projection of $x(t)$ onto
$\operatorname{span}(\varphi_{01},\ldots,\varphi_{0N})$.
If the null controls and their corresponding outputs belong to $L^2$, this reconstruction is given explicitly by
\begin{equation}
w_N(t)
=
\sum_{j=1}^N
\left(
\int_0^T
\langle u(t-\tau),\mu_j(\tau)\rangle_U
-
\langle y(t-\tau),\eta_j(\tau)\rangle_Y
\,{\rm d}\tau
\right)
\varphi_{0j}.
\label{eq:staterecons}
\end{equation}

We conclude this section with two remarks on the state-reconstruction procedure.

\begin{rem}\
\begin{enumerate}[(a)]
\item
The null control $\eta_j$ associated with $\varphi_{0j}$ is generally not unique. If null controls exist in the classical function-space setting, one possible choice is an optimal control obtained by solving
\[
\begin{aligned}
\text{minimise}
&\quad
\frac{1}{2}
\int_0^T
\left(
\|\eta_j(t)\|_Y^2+\|\mu_j(t)\|_U^2
\right)
\,{\rm d}t
\\
\text{subject to}
&\quad
\spvek{\dot{\varphi}_j(t)}{\mu_j(t)}
=
\sbvek{A^d\&B^d}{C^d\&D^d}
\spvek{\varphi_j(t)}{\eta_j(t)},
\qquad
\varphi_j(0)=\varphi_{0j},
\qquad
\varphi_j(T)=0.
\end{aligned}
\]
According to \cite[Sec.~5]{ReSc24}, the optimal control is characterised by the boundary-value problem
\begin{align*}
\spvek{\dot{\varphi}_j(t)}{\mu_j(t)}
&=
\sbvek{A^d\&B^d}{C^d\&D^d}
\spvek{\varphi_j(t)}{\eta_j(t)},
\qquad
\varphi_j(0)=\varphi_{0j},
\qquad
\varphi_j(T)=0,
\\
\spvek{\dot{x}_j(t)}{\eta_j(t)}
&=
-\sbvek{A\&B\\[-1mm]}{C\&D}
\spvek{x_j(t)}{\mu_j(t)}.
\end{align*}

\item
The existence of a dense subspace $D\subset X$ consisting of states that are distributionally null-controllable in time $T$ by the adjoint system \eqref{eq:ODEnodeadj} implies that the system \eqref{eq:ODEnode} is {\em reconstructable in time $T$}. More precisely, for every $t\geq T$, the input and output on $[0,t]$ uniquely determine the state $x(t)$. This property is also referred to as {\em final-state observability}; see \cite{KS2023}.

The corresponding property of the adjoint system may be regarded as {\em approximate null-con\-trollability in time $T$}, since every state can be approximated by states that are null-controllable in time $T$. It generalises {\em exact null-con\-trolla\-bility in time $T$}, which requires every state to be null-controllable in time $T$.

For well-posed systems, exact null-controllability of the adjoint system is equivalent to {\em exact reconstructability of \eqref{eq:ODEnode} in time $T$}; see \cite[Sec.~9.4]{Staffans2005}. Here, exact reconstructability in time $T$ means that there exists some $M>0$ such that every classical, and hence every generalised, trajectory of \eqref{eq:ODEnode} with $u\equiv0$ satisfies
\[
\|x(T)\|_X
\leq
M\|y\|_{L^2([0,T];Y)}.
\]
The above observation establishes a duality between reconstructability of \eqref{eq:ODEnode} and approximate null-controllability of the adjoint system \eqref{eq:ODEnodeadj}, without requiring the system to be well posed.
\end{enumerate}
\end{rem}

\section{State feedback}\label{sec:feedback}

We now apply Theorem~\ref{thm:main} to the implementation of state feedback. To this end, we introduce a subspace of $X$ consisting of those states that are compatible with some input in the domain of $A\&B$. More precisely, let
\begin{subequations}\label{eq:Vspace}
\begin{equation}
\mathcal{V}
=
\setdef{x\in X}{\exists\,u\in U\text{ such that }\spvek{x}{u}\in\dom(A\&B)}.
\end{equation}
Equipped with the norm
\begin{equation}\label{eq:Vspacenorm}
\|x\|_{\mathcal{V}}
=
\inf
\setdef{
\left\|\spvek{x}{u}\right\|_{\dom(A\&B)}
}{\,
u\in U\text{ such that }\spvek{x}{u}\in\dom(A\&B)
},
\end{equation}
\end{subequations}
the space $\mathcal{V}$ is a Hilbert space; see
\cite[Lem.~4.3.12]{Staffans2005}. Moreover, the embeddings
\[
\dom(A)\subset\mathcal{V}\subset X
\]
are continuous.

The feedback operators considered below need not be bounded on $X$. Instead, they are assumed to be bounded from $\mathcal{V}$ to the input space $U$. Recall from Remark~\ref{rem:nodes}\,\eqref{rem:nodesb} that $A\&B$ extends to
\[
[A_{-1}\ B]\in L(X\times U,X_{-1}),
\]
where
$A_{-1}:X_{-1}\supset\dom(A_{-1})=X\to X_{-1}$
is the extension of $A$.

We first specify what is meant by an admissible state feedback for the control system \eqref{eq:ODE}.

\begin{defn}[Admissible state feedback]
Let $X$ and $U$ be Hilbert spaces, and assume that
$A\&B:\dom(A\&B)\subset X\times U\to X$
satisfies conditions \eqref{def:sysnodea}, \eqref{def:sysnodec}, and
\eqref{def:sysnoded} of Definition~\ref{def:sysnode}.
Let $\mathcal{V}$ be defined as in \eqref{eq:Vspace}, and let
$[A_{-1}\ B]\in L(X\times U,X_{-1})$
be the extension of $A\&B$ introduced in
Remark~\ref{rem:nodes}\,\eqref{rem:nodesb}.
An operator $F\in L(\mathcal{V},U)$ is called an
{\em admissible state feedback} for
$\dot{x}(t)=A\&B\spvek{x(t)}{u(t)}$
if the operator
$A_F:\dom(A_F)\subset X\to X$ defined by
\begin{align*}
\dom(A_F)
&=
\setdef{x\in\mathcal{V}}{A_{-1}x+BFx\in X},
\\
A_Fx
&=
A_{-1}x+BFx
\end{align*}
generates a strongly continuous semigroup on $X$.
\end{defn}

For further details on state feedback for system nodes, we refer to
\cite[Sec.~7.3]{Staffans2005}.

To apply Theorem~\ref{thm:main}, we assume from now on that the input space is finite-dimensional. This assumption is natural for systems with finitely many actuators. After identifying $U$ with $\R^m$, an admissible feedback operator
$F\in L(\mathcal{V},\R^m)$
can be represented by an $m$-tuple
\[
(F_1,\ldots,F_m)\in(\mathcal{V}^\prime)^m,
\qquad
F_i=e_i^\top F,\quad i=1,\ldots,m,
\]
where $e_i\in\R^m$ denotes the $i$th canonical unit vector. Consequently,
\begin{equation}
Fx
=
\sum_{i=1}^m
\langle x,F_i\rangle_{\mathcal{V},\mathcal{V}^\prime}\,e_i,
\qquad x\in\mathcal{V}.
\label{eq:Feeddec}
\end{equation}

Let $T>0$. For each $i=1,\ldots,m$, define
$(K\&L)_i\in\dom(A\&B)^\prime$ by
\[
(K\&L)_i\spvek{x}{u}
=
\langle x,F_i\rangle_{\mathcal{V},\mathcal{V}^\prime}
\qquad
\forall\,\spvek{x}{u}\in\dom(A\&B).
\]
Suppose that, for each $i=1,\ldots,m$, there exists a generalised null control
\[
\eta_i\in H^{-k_i}_{0,\loc}(\R_{\geq0};Y),
\qquad k_i\in\N_0,
\]
in time $T$ for the extended adjoint system corresponding to
$K\&L=(K\&L)_i$, and let
\[
\mu_i\in H^{-k_i-2}_{0,\loc}(\R_{\geq0};U)
\]
denote the corresponding output.

Let $(x,u,y)$ be a generalised trajectory of \eqref{eq:ODEnode}. For
$i=1,\ldots,m$, define the distributional additional output
\[
z_i
\coloneqq 
(K\&L)_i\spvek{x}{u}
\in H^{-2}_{0,\loc}(\R_{\geq0}).
\]
By Theorem~\ref{thm:main},
\[
\supp\big(u\ast\mu_i-y\ast\eta_i-z_i\big)
\subset[0,T],
\qquad i=1,\ldots,m.
\]
If $\spvek{x}{u}$ is represented by a function with values in
$\dom(A\&B)$, then
\[
z_i(t)
=
\langle x(t),F_i\rangle_{\mathcal{V},\mathcal{V}^\prime}
\]
for almost every $t\geq0$. In particular, this identity holds for
classical trajectories.
Accordingly, the distributional extension of the feedback output is
\[
Fx
:=
\sum_{i=1}^m z_i e_i
\in H^{-2}_{0,\loc}(\R_{\geq0};U).
\]
Combining the above reconstruction identities yields
\[
Fx
=
\sum_{i=1}^m
\big(u\ast\mu_i-y\ast\eta_i\big)e_i
\qquad
\text{on }(T,\infty).
\]
Consequently, the state feedback $u=Fx$ can be realised by imposing
\[
u
=
\sum_{i=1}^m
\big(u\ast\mu_i-y\ast\eta_i\big)e_i
\qquad
\text{on }(T,\infty).
\]
For classical trajectories, this relation realises the state feedback
$u(t)=Fx(t)$ pointwise. For generalised trajectories, it is understood
in the distributional sense described above.

\section{Examples}\label{sec:ex}

\subsection{Stabilisation of a vibrating string}

Consider a spatially homogeneous, undamped vibrating string of length
$\ell>0$ that is clamped at the right endpoint and force-controlled at
the left endpoint. For simplicity, we assume that the propagation speed
is equal to one. The velocity of the string is measured at an interior point $\xi_0\in(0,\ell)$. We assume that 
\[ \frac{\xi_0}{\ell}=\frac{p}{q}, \qquad p,q\in\N,\qquad \gcd(p,q)=1, \]
where at least one of $p$ and $q$ is even.

Let $w(t,\xi)$ denote the displacement at position
$\xi\in[0,\ell]$ and time $t\geq0$, and abbreviate
$v^\prime\coloneqq \tfrac{\partial v}{\partial\xi}$. The system is governed by
\begin{equation}
\begin{aligned}
\ddot{w}(t,\xi)
&=
w^{\prime\prime}(t,\xi),
&
(t,\xi)
&\in\R_{\geq0}\times[0,\ell],
\\
u(t)
&=
-w^\prime(t,0),
\qquad
0=w(t,\ell),
&
t
&\in\R_{\geq0},
\\
y(t)
&=
\dot{w}(t,\xi_0),
&
t
&\in\R_{\geq0}.
\end{aligned}
\label{eq:waveq}
\end{equation}
The configuration is illustrated in Fig.~\ref{fig:example_wave}. The
equation is supplemented with initial conditions for the displacement
and velocity, which need not be specified here.

\begin{figure}[htb]
	\centering	
		\begin{tikzpicture}[scale=0.8]
			\pgfmathsetmacro{\fl}{0}
			\pgfmathsetmacro{\fr}{10}
			\pgfmathsetmacro{\fb}{0}
			\pgfmathsetmacro{\ft}{5}
			\pgfmathsetmacro{\fbb}{-1.5}
			\pgfmathsetmacro{\ca}{0.5}
			\pgfmathsetmacro{\cb}{1.0}
			
			\draw[-,color = black] (\fl,\fb) to (\fr,\fb);
			\draw[-{Latex[length=3mm]},color = black, thick] (\fl,\fb) to (\fl+2cm,\fb);
            \draw[color = black,thick] (\fl,\fb+0.2cm) to (\fl,\fb-0.2cm);
			
			\node[isosceles triangle,
			isosceles triangle apex angle=60,
			draw,fill=white, thick, rotate=90,
			minimum size =0.5cm, anchor=apex] (T1) at (\fr,\fb){};
			\draw[very thick] ([xshift=-0.2cm]T1.left corner) to ([xshift=0.2cm]T1.right corner);
			\draw[pattern=north east lines, thick, draw=none] ([xshift=-0.2cm]T1.left corner) rectangle ([xshift=0.2cm,yshift=-0.2cm]T1.right corner);
			
			\node[isosceles triangle,
			isosceles triangle apex angle=60,
			draw,fill=white, thick,
			minimum size =0.5cm, anchor=apex] (T2) at (\fl,\fb+0.5){};
			\draw[very thick] ([yshift=-0.2cm]T2.right corner) to ([yshift=0.2cm]T2.left corner);
			\draw[very thick] ([xshift=-0.2cm,yshift=-0.2cm]T2.right corner) to ([xshift=-0.2cm,yshift=0.2cm]T2.left corner);
			\draw[pattern=north east lines, thick, draw=none] ([xshift=-0.2cm,yshift=-0.2cm]T2.right corner) rectangle ([xshift=-0.4cm,yshift=0.2cm]T2.left corner);
			
			\draw[dashed,color=black, thick] (T2.apex) to ([xshift=1.7cm,yshift=1.24cm]T2.apex);
			\draw[color=black,thick]
			([xshift=1.0cm,yshift=1.45cm]T2.apex)
			node[anchor=south]{$u(t)=-w^\prime(t,0)$};
			\draw[color=black,thick]
			([xshift=1.0cm]T1.apex)
			node[anchor=south]{$w(t,\ell)=0$};
			\draw[-,color=black, ultra thick] (T1.apex) .. controls (7cm,-3cm) and (3cm,3cm) .. (T2.apex);
			
			\draw[dashed,color=black, thick] (\fl+4cm,\fb+0.85cm) to (\fl+4cm,\fb-0.1cm);
			\draw[color=black, dotted, fill=white, very thick] (\fl+4cm,\fb+0.52cm) circle (3pt);
			\draw[color=black, thick]
			(\fl+4cm,\fb-0.1cm) node[anchor=north]{$\xi_0$}
			(\fl+4cm,\fb+0.9cm) node[anchor=south]{$y(t)=\dot{w}(t,\xi_0)$};
			
			\draw[-{Latex[length=3mm]},color=black, thick] (\fl+7,\fb) to (\fl+7,\fb-0.75);
			\draw[color=black, thick]
			(\fl+7,\fb-0.4) node[anchor=west]{$w(t,\xi)$};
			
			\draw[color=black, fill=white, very thick] (T1.apex) circle (3pt);
			\draw[color=black, fill=white, very thick] (T2.apex) circle (3pt);
			
			\draw[color=black, thick]
			(\fl+1.7cm,\fb) node[anchor=south]{$\xi$};
		\end{tikzpicture}
\vspace*{-2cm}
        
\caption{Vibrating string with force input at $\xi=0$ and velocity output at $\xi_0\in(0,\ell)$.}
\label{fig:example_wave}
\end{figure}

Our objective is to stabilise the system exponentially. This can be
achieved by applying velocity feedback at the left endpoint, that is,
by using the control law
\[
u(t)=-k\dot{w}(t,0)
\]
with $k>0$. According to \cite[Thm.~9.1.3]{JaZw12}, this feedback law
renders the system exponentially stable. Since only the velocity at the
interior point $\xi_0$ is measured, however, the boundary velocity
$\dot{w}(t,0)$ is not directly available. We therefore apply the
framework developed above to reconstruct this quantity from the input
and output.

We first represent \eqref{eq:waveq} as a system node. As is customary
for the wave equation, we introduce the state
\[
x(t)
=
\left(
\begin{smallmatrix}
q(t)\\
p(t)
\end{smallmatrix}
\right)
\coloneqq 
\left(
\begin{smallmatrix}
w^\prime(t,\cdot)\\
\dot{w}(t,\cdot)
\end{smallmatrix}
\right)
\in L^2([0,\ell];\R^2).
\]
Here, $q(t),p(t)\in L^2([0,\ell])$ denote the spatial distributions of
the strain and momentum, respectively. In the present dimensionless
setting, the momentum coincides with the velocity. Since the clamping
condition $w(t,\ell)=0$ implies $p(t,\ell)=0$, the system can be written
in the form \eqref{eq:ODEnode} with
$S=\sbvek{A\&B\\[-1mm]}{C\&D}$, where
\begin{subequations}\label{eq:wavenode}
\begin{equation}
\begin{aligned}
\dom(S)=\dom(A\&B)
&=
\setdef{
\left(
\begin{smallmatrix}
q\\
p\\
u
\end{smallmatrix}
\right)
\in H^1([0,\ell];\R^2)\times\R
}{
p(\ell)=0
\ \wedge\
-q(0)=u
},
\\
A\&B
\left(
\begin{smallmatrix}
q\\
p\\
u
\end{smallmatrix}
\right)
&=
\left(
\begin{smallmatrix}
p^\prime\\
q^\prime
\end{smallmatrix}
\right),
\qquad
C\&D
\left(
\begin{smallmatrix}
q\\
p\\
u
\end{smallmatrix}
\right)
=
p(\xi_0).
\end{aligned}
\end{equation}
The fact that this defines a system node on
$(L^2([0,\ell];\R^2),\R,\R)$ can be established using the techniques
developed in \cite[Sec.~4.2]{PhReSc23}. The class considered there
differs only in that the output is a boundary value. Replacing this
boundary output by the interior point evaluation above does not affect
the system-node property.

To reconstruct the velocity at the left endpoint, we introduce the
additional output \eqref{eq:zout} determined by
$K\&L\in\dom(A\&B)^\prime$ with
\begin{equation}
K\&L
\left(
\begin{smallmatrix}
q\\
p\\
u
\end{smallmatrix}
\right)
=
p(0).
\end{equation}
\end{subequations}

We next determine the adjoint of the extended system node
\[
S_{\rm ext}
=
\left[
\begin{smallmatrix}
A\&B\\
C\&D\\
K\&L
\end{smallmatrix}
\right].
\]
Let
$\varphi_q,\varphi_p\in L^2([0,\ell])$ and
$\eta,\alpha\in\R$. Then
\begin{equation}\label{eq:waveadj}
\begin{aligned}
&
\left\langle
\left(
\begin{smallmatrix}
\varphi_q\\
\varphi_p\\
\eta\\
\alpha
\end{smallmatrix}
\right),
\left[
\begin{smallmatrix}
A\&B\\
C\&D\\
K\&L
\end{smallmatrix}
\right]
\left(
\begin{smallmatrix}
q\\
p\\
u
\end{smallmatrix}
\right)
\right\rangle
\\
&\qquad
=
\left\langle
\left(
\begin{smallmatrix}
\varphi_q\\
\varphi_p\\
\eta\\
\alpha
\end{smallmatrix}
\right),
\left(
\begin{smallmatrix}
p^\prime\\
q^\prime\\
p(\xi_0)\\
p(0)
\end{smallmatrix}
\right)
\right\rangle_{L^2([0,\ell];\R^2)\times\R^2}
\\
&\qquad
=
\int_0^\ell
\varphi_q(\xi)p^\prime(\xi)
+
\varphi_p(\xi)q^\prime(\xi)
\,{\rm d}\xi
+
\eta p(\xi_0)
+
\alpha p(0).
\end{aligned}
\end{equation}
Assume that
\[
\left(
\begin{smallmatrix}
\varphi_q\\
\varphi_p\\
\eta\\
\alpha
\end{smallmatrix}
\right)
\in\dom(S_{\rm ext}^*).
\]
Choosing $p,q\in C_0^\infty([0,\ell])$ with
$\xi_0\notin\supp(p)$ and using the definition of the weak derivative
gives
\begin{equation}
\varphi_p\in H^1([0,\ell]),
\qquad
\varphi_q|_{[0,\xi_0]}\in H^1([0,\xi_0]),
\qquad
\varphi_q|_{[\xi_0,\ell]}\in H^1([\xi_0,\ell]).
\label{eq:phismooth}
\end{equation}
We denote the left and right limits of a function $f$ at $\xi_0$ by
$f(\xi_0^-)$ and $f(\xi_0^+)$, respectively. Moreover,
$\varphi_q^\prime\in L^2([0,\ell])$ denotes the function that coincides
with the weak derivative of $\varphi_q|_{[0,\xi_0]}$ on
$[0,\xi_0]$ and with that of $\varphi_q|_{[\xi_0,\ell]}$ on
$[\xi_0,\ell]$.

Using $p(\ell)=0$ and $q(0)=-u$, the right-hand side of
\eqref{eq:waveadj} becomes
\begin{align*}
&
-\int_0^{\xi_0}
\varphi_q^\prime(\xi)p(\xi)\,{\rm d}\xi
+
\varphi_qp\Big|_0^{\xi_0^-}
-
\int_{\xi_0}^\ell
\varphi_q^\prime(\xi)p(\xi)\,{\rm d}\xi
+
\varphi_qp\Big|_{\xi_0^+}^{\ell}
\\
&\qquad
-
\int_0^\ell
\varphi_p^\prime(\xi)q(\xi)\,{\rm d}\xi
+
\varphi_pq\Big|_0^\ell
+
\eta p(\xi_0)
+
\alpha p(0)
\\
&=
-\int_0^\ell
\big(
\varphi_q^\prime(\xi)p(\xi)
+
\varphi_p^\prime(\xi)q(\xi)
\big)
\,{\rm d}\xi
\\
&\qquad
+
\big(
\varphi_q(\xi_0^-)
-
\varphi_q(\xi_0^+)
+
\eta
\big)p(\xi_0)
+
\big(
\alpha-\varphi_q(0)
\big)p(0)
+
\varphi_p(\ell)q(\ell)
+
\varphi_p(0)u.
\end{align*}
Consequently,
\begin{align*}
\dom(S_{\rm ext}^*)
&=
\setdef{
\left(
\begin{smallmatrix}
\varphi_q\\
\varphi_p\\
\eta\\
\alpha
\end{smallmatrix}
\right)
\in L^2([0,\ell];\R^2)\times\R^2
}{
\begin{array}{l}
\eqref{eq:phismooth}\text{ holds},\
\varphi_p(\ell)=0,\\
\alpha=\varphi_q(0),\
\eta=\varphi_q(\xi_0^+)-\varphi_q(\xi_0^-)
\end{array}
},
\\
S_{\rm ext}^*
\left(
\begin{smallmatrix}
\varphi_q\\
\varphi_p\\
\eta\\
\alpha
\end{smallmatrix}
\right)
&=
\left(
\begin{smallmatrix}
-\varphi_p^\prime\\
-\varphi_q^\prime\\
\varphi_p(0)
\end{smallmatrix}
\right).
\end{align*}

Thus, the system governed by $S_{\rm ext}^*$ is again a wave equation,
with a clamped condition at the right endpoint and force control at the
left endpoint, together with an interface condition at $\xi_0$. Define
$\varphi_w:\R_{\geq0}\to H^1([0,\ell])$ by
\begin{equation}
\varphi_w(t,\xi)
=
\int_\xi^\ell
\varphi_q(t,\zeta)\,{\rm d}\zeta.
\label{eq:wadj}
\end{equation}
The adjoint system can then be written as
\begin{equation}
\begin{aligned}
\ddot{\varphi}_w(t,\xi) &=\phantom{-}\varphi_w^{\prime\prime}(t,\xi), && (t,\xi)\in\R_{\geq0}\times \big((0,\xi_0)\cup(\xi_0,\ell)\big),
\\
\alpha(t)
&=-\varphi_w^\prime(t,0),
&& t\in\R_{\geq0},
\\
\eta(t)
&=-\varphi_w^\prime(t,\xi_0^+)
  +\varphi_w^\prime(t,\xi_0^-),
&& t\in\R_{\geq0},
\\
0
&=\phantom{-}\varphi_w(t,\ell),
&& t\in\R_{\geq0},
\\
\mu(t)
&=\phantom{-}\dot{\varphi}_w(t,0),
&& t\in\R_{\geq0}.
\end{aligned}
\label{eq:waveqadj}
\end{equation}
The input $\eta$ may therefore be interpreted as a point force acting
at $\xi_0$.

We now construct a generalised null control. Set
$h\coloneqq \frac{\ell}{q}$,
so that
$\xi_0=ph$, $\ell=qh$.
We use the right-shift operator $S_{r,h}$ introduced in
Section~\ref{sec:prelim}. Thus, $S_{r,h}^j$ represents a delay of length $jh$. In particular, $S_{r,h}^j\delta=\delta_{jh}$. For a polynomial $P(X)=p_0+p_1X+\cdots+ p_NX^N$, we use the notation \[ P(S_{r,h}) \coloneqq  \sum_{j=0}^N p_jS_{r,h}^j. \]
Since $p$ and $q$ are coprime and at least one of them is
even, the polynomials
$1+X^{2p}$, $1+X^{2q}$ are coprime. Hence, there exist polynomials
$P,Q\in\R[X]$ satisfying the Bézout identity
\begin{equation}
P(X)(1+X^{2p})+Q(X)(1+X^{2q})=1.
\label{eq:wavebezout}
\end{equation}
To determine a corresponding null control, introduce the Riemann
invariants
\[
R\coloneqq \dot{\varphi}_w+\varphi_w^\prime,
\qquad
S\coloneqq \dot{\varphi}_w-\varphi_w^\prime.
\]
On both subintervals $(0,\xi_0)$ and $(\xi_0,\ell)$, they satisfy
\[
\dot R-R^\prime=0,
\qquad
\dot S+S^\prime=0.
\]
Thus, $R$ and $S$ represent left- and right-travelling waves,
respectively.
Consider a smooth input $\alpha\in C_0^\infty(\R_{>0})$ and zero
initial conditions. At the interface, define
\[
\begin{aligned}
s_-(t)&\coloneqq S(t,\xi_0^-),
&
r_+(t)&\coloneqq R(t,\xi_0^+),
\\
r_-(t)&\coloneqq R(t,\xi_0^-),
&
s_+(t)&\coloneqq S(t,\xi_0^+).
\end{aligned}
\]
Continuity of $\dot{\varphi}_w$ and the interface condition in
\eqref{eq:waveqadj} yield
\begin{equation}
r_-=r_++\eta,
\qquad
s_+=s_-+\eta.
\label{eq:waveinterface}
\end{equation}
Propagation between the interface and the two endpoints gives
\begin{equation}
s_-
=
S_{r,h}^{2p}r_-
+
2S_{r,h}^p\alpha,
\qquad
r_+
=
-S_{r,h}^{2(q-p)}s_+.
\label{eq:wavepropagation}
\end{equation}
Indeed, the first identity describes propagation to the left endpoint
and back, together with the boundary input $\alpha$, whereas the
second identity incorporates the sign change caused by reflection at
the clamped endpoint.
Combining \eqref{eq:waveinterface} and
\eqref{eq:wavepropagation}, we obtain
\begin{equation}
\big(\Id+S_{r,h}^{2q}\big)s_+
=
2S_{r,h}^p\alpha
+
\big(\Id+S_{r,h}^{2p}\big)\eta.
\label{eq:wavecharrelation}
\end{equation}
Motivated by the Bézout identity \eqref{eq:wavebezout}, define
\begin{equation}
\eta_\alpha
\coloneqq 
-2P(S_{r,h})S_{r,h}^p\alpha.
\label{eq:waveadjinp}
\end{equation}
Then \eqref{eq:wavecharrelation} and \eqref{eq:wavebezout} give
\[
\big(\Id+S_{r,h}^{2q}\big)s_+
=
2\big(\Id+S_{r,h}^{2q}\big)
Q(S_{r,h})S_{r,h}^p\alpha.
\]
Since $\Id+S_{r,h}^{2q}$ is injective on causal distributions, it
follows that
\[
s_+
=
2Q(S_{r,h})S_{r,h}^p\alpha.
\]
The remaining travelling-wave components are therefore
\[
\begin{aligned}
r_+
&=
-2S_{r,h}^{2q-p}Q(S_{r,h})\alpha,
\\
r_-
&=
-2S_{r,h}^p
\big(
P(S_{r,h})
+
S_{r,h}^{2(q-p)}Q(S_{r,h})
\big)\alpha,
\\
s_-
&=
2S_{r,h}^p
\big(
P(S_{r,h})+Q(S_{r,h})
\big)\alpha.
\end{aligned}
\]
Consequently, all travelling-wave components are finite linear
combinations of delayed copies of $\alpha$. In particular, the
corresponding adjoint state has compact temporal support.

At the left endpoint,
\[
\alpha=\tfrac{1}{2}(S-R),
\qquad
\mu=\tfrac{1}{2}(S+R),
\]
and the incoming left-travelling wave is given by
$S_{r,h}^p r_-$. Hence,
\[
\mu_\alpha
=
\alpha+S_{r,h}^p r_-.
\]
Using \eqref{eq:wavebezout}, we obtain
\begin{equation}
\mu_\alpha
=
A(S_{r,h})\alpha,
\qquad
A(X)
\coloneqq 
P(X)(1-X^{2p})+Q(X)(1-X^{2q}).
\label{eq:waveadjout}
\end{equation}
Equivalently,
\[
A(X)
=
1-2X^{2p}P(X)-2X^{2q}Q(X).
\]

Let $(\alpha_n)$ be a one-sided mollifier sequence satisfying
\[
\supp\alpha_n\subset(0,h/n],
\qquad
\alpha_n\longrightarrow\delta
\quad\text{in }H^{-1}_{0,\loc}(\R_{\geq0}).
\]
Define $\eta_n
\coloneqq 
-2P(S_{r,h})S_{r,h}^p\alpha_n$, $\mu_n
\coloneqq 
A(S_{r,h})\alpha_n$,
and let $\varphi_n$ be the corresponding adjoint state. Since all
travelling-wave components consist of finitely many fixed delays of
$\alpha_n$, there exists some $T>0$, independent of $n$, such that
\[
\supp\varphi_n\subset[0,T],
\qquad
\supp\eta_n\subset[0,T],
\qquad
\supp\mu_n\subset[0,T].
\]
Passing to the distributional limit gives
\[
\eta
=
-2P(S_{r,h})S_{r,h}^p\delta,
\qquad
\mu
=
A(S_{r,h})\delta,
\]
together with a corresponding adjoint state $\varphi$ satisfying $\supp\varphi\subset[0,T]$.
Thus, $\eta$ is a generalised null control satisfying the assumptions
of Theorem~\ref{thm:main}, for instance with $k=1$.
Consequently, for every generalised trajectory of \eqref{eq:waveq},
understood as a trajectory of the system node \eqref{eq:wavenode}, the
distribution
$u\ast\mu-y\ast\eta$
coincides on $(T,\infty)$ with the velocity at the left endpoint.
Therefore,
\begin{equation}
\dot w(\cdot,0)
=
A(S_{r,h})u
+
2P(S_{r,h})S_{r,h}^p y
\qquad
\text{on }(T,\infty).
\label{eq:wavevelocityreconstruction}
\end{equation}
The Bézout identity \eqref{eq:wavebezout} evaluated at $X=0$ gives
\[
P(0)+Q(0)=1,
\]
and hence $A(0)=1$. We may therefore write
\[
A(X)=1+\widetilde A(X),
\qquad
\widetilde A(0)=0.
\]
The exponentially stabilising feedback law
\[
u(t)=-k\dot w(t,0),
\qquad
k>0,
\]
can thus be realised, after time $T$, in the causal form
\begin{equation}
u
=
-\frac{k}{1+k}
\left(
\widetilde A(S_{r,h})u
+
2P(S_{r,h})S_{r,h}^p y
\right).
\label{eq:wavefeedbackrealisation}
\end{equation}
Since $\widetilde A(0)=0$ and $p\geq1$, the right-hand side of
\eqref{eq:wavefeedbackrealisation} contains only delayed values of the
input and output.

For example, if $\xi_0=\ell/2$, then $p=1$ and $q=2$. One possible
choice in \eqref{eq:wavebezout} is
$P(X)=\tfrac12-\tfrac12 X^2$,
$Q(X)=\tfrac{1}{2}$.
In this case,
$\eta=-\delta_h+\delta_{3h}$, $\mu=\delta-\delta_{2h}$.
The corresponding adjoint state is supported in $[0,3h]$. Hence,
Theorem~\ref{thm:main} applies with
$T=3h=\tfrac{3\ell}{2}$,
and \eqref{eq:wavevelocityreconstruction} becomes
\[
\dot w(t,0)
=
u(t)-u(t-2h)+y(t-h)-y(t-3h),
\qquad
t>3h.
\]
Consequently, the stabilising feedback law
$u(t)=-k\dot w(t,0)$ can be implemented as
\[
u(t)
=
\frac{k}{1+k}
\left(
u(t-\ell)
-y\left(t-\tfrac{\ell}{2}\right)
+y\left(t-\tfrac{3\ell}{2}\right)
\right),
\qquad
t>\tfrac{3\ell}{2}.
\]

\subsection{State reconstruction for a reaction--diffusion equation with Dirichlet boundary control}

Consider a reaction--diffusion equation on the rectangular domain
$\Omega=(0,L_1)\times(0,L_2)\subset\R^2$. The control acts through the
Dirichlet boundary values on the right-hand part of the boundary,
\[
\Gamma=\{L_1\}\times[0,L_2]\subset\partial\Omega,
\]
while homogeneous Dirichlet boundary conditions are imposed on
$\partial\Omega\setminus\Gamma$. The output is given by the negative Neumann trace on $\Gamma$. We consider a unit diffusion coefficient and a spatially constant reaction coefficient $c\in\R$. Thus, the system is governed by
\begin{equation}\label{eq:rdeq}
\begin{aligned}
    \dot{x}(t,\xi)
    &=\phantom{-}
    \Delta x(t,\xi)+cx(t,\xi),
    &(t,\xi)&\in\R_{\geq0}\times\Omega,
    \\
    u(t,\xi)
    &=\phantom{-}
    x(t,\xi),
    &(t,\xi)&\in\R_{\geq0}\times\Gamma,
    \\
    0
    &=\phantom{-}
    x(t,\xi),
    &(t,\xi)&\in\R_{\geq0}\times
    \big(\partial\Omega\setminus\Gamma\big),
    \\
    y(t,\xi)
    &=-
    n^\top(\xi)\nabla x(t,\xi),
    &(t,\xi)&\in\R_{\geq0}\times\Gamma,
\end{aligned}
\end{equation}
where $n\in L^\infty(\partial\Omega;\R^2)$ denotes the outward unit
normal vector field on $\partial\Omega$.

We next formulate \eqref{eq:rdeq} as a system node. The state space is
$L^2(\Omega)$. Since the full Dirichlet trace belongs to the fractional
Sobolev space $H^{1/2}(\partial\Omega)$; see
\cite[Sec.~4.1]{PhReSc23}, the input space is chosen as
$H^{1/2}_0(\Gamma)$. This is the space of all elements of
$H^{1/2}(\Gamma)$ whose extension by zero to
$\partial\Omega\setminus\Gamma$ belongs to
$H^{1/2}(\partial\Omega)$; see also \cite{ReSc23a}. The output space is
\[
H^{-1/2}(\Gamma)\coloneqq H^{1/2}_0(\Gamma)^\prime.
\]
Let
$\gamma:H^1(\Omega)\to H^{1/2}(\partial\Omega)$
denote the {\em trace operator}, and define
\[
H^1_\Gamma(\Omega)
=
\setdef{x\in H^1(\Omega)}
{(\gamma x)|_{\partial\Omega\setminus\Gamma}=0}.
\]
The restricted trace operator
\[
\gamma_\Gamma:H^1_\Gamma(\Omega)\to H^{1/2}_0(\Gamma)
\]
maps a function to the restriction of its trace to $\Gamma$.
Furthermore, let $H_{\divg}(\Omega)$ denote the space of all vector
fields in $L^2(\Omega;\R^2)$ whose weak divergence belongs to
$L^2(\Omega)$. The restricted normal trace operator is denoted by
\[
\gamma_{\Gamma,n}:H_{\divg}(\Omega)\to H^{-1/2}(\Gamma);
\]
see \cite{ReSc23a}. For the basic properties of the normal trace
operator, we refer to \cite{Tart07}.

With these spaces and trace operators, the system node associated with
\eqref{eq:rdeq} is given by
\begin{equation}\label{eq:rdeqnode}
\begin{aligned}
\dom S=\dom(A\&B)
&=
\setdef{
\spvek{x}{u}\in H^1_\Gamma(\Omega)\times H^{1/2}_0(\Gamma)
}{
\nabla x\in H_{\divg}(\Omega)
\ \wedge\
u=\gamma_\Gamma x
},
\\
A\&B\spvek{x}{u}
&=
\Delta x+cx,
\qquad
C\&D\spvek{x}{u}
=-
\gamma_{\Gamma,n}(\nabla x).
\end{aligned}
\end{equation}
A direct calculation shows that $S$ is a self-dual system node, that
is, $S=S^\prime$. Consequently, if
\[
\mathcal{R}:H^{1/2}_0(\Gamma)\to H^{-1/2}(\Gamma)
\]
denotes the Riesz isomorphism, then the adjoint of $S$ satisfies
\begin{equation}
S^*
=
\sbmat{\Id}{0}{0}{\mathcal{R}^{-1}}
S
\sbmat{\Id}{0}{0}{\mathcal{R}^{-1}}.
\label{eq:rdadj}
\end{equation}

We now apply the state-reconstruction procedure developed in
Section~\ref{sec:stateest}. To this end, we determine null controls for
the adjoint system.
By \eqref{eq:rdadj}, the adjoint differs from the
original system \eqref{eq:rdeq} only through the Riesz
identifications of its input and output spaces. More precisely, the
adjoint system is described by
\begin{equation}\label{eq:rdeq_adj}
\begin{aligned}
    \dot{\varphi}(t,\xi)
    &=\phantom{-}
    \Delta\varphi(t,\xi)+c\varphi(t,\xi),
    &(t,\xi)&\in\R_{\geq0}\times\Omega,
    \\
    \widetilde{\eta}(t,\xi)
    &=\phantom{-}
    \varphi(t,\xi),
    &(t,\xi)&\in\R_{\geq0}\times\Gamma,
    \\
    0
    &=\phantom{-}
    \varphi(t,\xi),
    &(t,\xi)&\in\R_{\geq0}\times
    \big(\partial\Omega\setminus\Gamma\big),
    \\
    \widetilde{\mu}(t,\xi)
    &=
    -n^\top(\xi)\nabla\varphi(t,\xi),
    &(t,\xi)&\in\R_{\geq0}\times\Gamma.
\end{aligned}
\end{equation}
The abstract input and output of the adjoint system are related to
these traces by
\begin{equation}\label{eq:rcriesz}
\eta(\cdot,t)
=
\mathcal{R}\widetilde{\eta}(\cdot,t),
\qquad
\mu(\cdot,t)
=
\mathcal{R}^{-1}\widetilde{\mu}(\cdot,t),
\qquad
t\in\R_{\geq0}.
\end{equation}
The results of \cite{LeRo95,LinGu95} imply that, for every $T>0$ and
$\varphi_0\in L^2(\Omega)$, there exists a null control
\[
\widetilde{\eta}
\in
L^2([0,T];H^{1/2}_0(\Gamma))
\]
such that the solution of \eqref{eq:rdeq_adj} satisfying
$\varphi(0,\cdot)=\varphi_0$ fulfils
$\varphi(T,\cdot)=0$. The corresponding output satisfies
\[
\widetilde{\mu}
\in
L^2([0,T];H^{-1/2}(\Gamma)).
\]
Consequently,
\[
\eta\in L^2([0,T];H^{-1/2}(\Gamma)),
\qquad
\mu\in L^2([0,T];H^{1/2}_0(\Gamma)).
\]
Thus, in this case, the reconstruction can be formulated entirely in
terms of functions rather than distributions.

The results of Section~\ref{sec:stateest} then imply that every
generalised, and hence every classical, trajectory of
\eqref{eq:rdeq} satisfies
\begin{equation}
\forall\,t\geq T:\quad
\langle x(t),\varphi_0\rangle_{L^2}
=
\int_0^T
\langle u(t-\tau),\mu(\tau)\rangle_{H^{1/2}}
-
\langle y(t-\tau),\eta(\tau)\rangle_{H^{-1/2}}
\,{\rm d}\tau.
\label{eq:concrd}
\end{equation}

The Riesz isomorphism can be eliminated from the final reconstruction
formula. Indeed, its definition gives
\[
\forall\,v\in H^{-1/2}(\Gamma),\
w\in H^{1/2}_0(\Gamma):\quad
\langle v,\mathcal{R}w\rangle_{H^{-1/2}}
=
\langle v,w\rangle_{H^{-1/2},H^{1/2}}
=
\langle\mathcal{R}^{-1}v,w\rangle_{H^{1/2}}.
\]
Hence, \eqref{eq:concrd} is equivalent to
\[
\forall\,t\geq T:\quad
\langle x(t),\varphi_0\rangle_{L^2}
=
\int_0^T
\langle u(t-\tau),\widetilde{\mu}(\tau)\rangle_{H^{1/2},H^{-1/2}}
-
\langle y(t-\tau),\widetilde{\eta}(\tau)\rangle_{H^{-1/2},H^{1/2}}
\,{\rm d}\tau.
\]
Thus, the two terms in the integrand are duality pairings between the
Dirichlet and Neumann traces of the original and adjoint systems.

For the numerical illustration, we compute null controls for the
adjoint system \eqref{eq:rdeq_adj} and reconstruct the projection of the
state using \eqref{eq:staterecons}, with the inner products in that
formula replaced by the corresponding duality pairings. The
orthonormal family $(\varphi_{0j})_{j\in\N}$ is obtained by applying the
Gram--Schmidt procedure to a basis of bivariate polynomials of total
degree at most ten. This results in an approximation space of dimension
$N=66$. The simulation results are shown in
Fig.~\ref{fig:HeatEq}.

\newlength\figH
\newlength\figW
\setlength{\figH}{7cm}
\setlength{\figW}{0.45\textwidth}
\begin{figure}
	\centering

	\caption{Top row: spatial domain $\Omega$ with input and output
	boundaries, and the $L^2$-norm of the estimation error
	$e(t,\xi)=x(t,\xi)-\widehat{x}(t,\xi)$.
	Bottom row: simulated state and estimation error
	$e(t,\xi_1,L_2/2)$ for $t\in[0,5]\,\mathrm{s}$.}
	\label{fig:HeatEq}
\end{figure}

\section*{Conclusion}

We have developed a modulating-function framework for partial state
reconstruction from past input and output data. The approach is based
on generalised null controls of the adjoint system and corresponding
convolutions with the input and output.

The first main contribution is a unified formulation for
infinite-dimensional linear systems in the system-node framework.
This framework requires comparatively few assumptions, is closed under
adjunction, and does not require the system to be well-posed. The
second contribution is a distributional extension of the
modulating-function method. The resulting modulating distributions
allow the reconstruction of additional outputs defined by unbounded
operators and are therefore particularly suitable for the
implementation of unbounded state feedback.

The theoretical results were illustrated by two examples. For a
vibrating string, the distributional approach was used to reconstruct a
boundary velocity from input and output data and thereby realise an
exponentially stabilising feedback law. For a reaction--diffusion
equation with Dirichlet boundary control and Neumann boundary
observation, the method was applied to finite-dimensional state
reconstruction and illustrated numerically.


\begin{thebibliography}{10}

\bibitem{adams2003sobolev}
R.~A. Adams and J.~J. Fournier.
\newblock {\em Sobolev spaces}.
\newblock Elsevier, 2003.

\bibitem{Alt16}
H.~Alt.
\newblock {\em Linear Functional Analysis, An Application-Oriented Introduction}.
\newblock Universitext. Springer London, 2016.

\bibitem{EngeNage00}
K.-J. Engel and R.~Nagel.
\newblock {\em One-parameter semigroups for linear evolution equations}, volume 194 of {\em Graduate Texts in Mathematics}.
\newblock Springer, New York, 2000.

\bibitem{FairS1970}
F.~W. Fairman and D.~W.~C. Shen.
\newblock Parameter identification for a class of distributed systems.
\newblock {\em International Journal of Control}, 11(6):929--940, 1970.

\bibitem{FJRS24}
B.~Farkas, B.~Jacob, T.~Reis, and M.~Schmitz.
\newblock Operator splitting based dynamic iteration for linear infinite-dimensional port-{H}amiltonian systems, 2024.
\newblock arXiv:2302.01195.

\bibitem{FiscD2016}
F.~Fischer and J.~Deutscher.
\newblock Algebraic fault detection and isolation for parabolic distributed–parameter systems using modulation functions.
\newblock In {\em 2nd IFAC Workshop on Control of Systems Governed by Partial Differential Equations}, pages 164 -- 169, 06 2016.

\bibitem{GhafNRLK2020}
L.~Ghaffour, M.~Noack, J.~Reger, and T.-M. Laleg-Kirati.
\newblock Non-asymptotic state estimation of linear reaction diffusion equation using modulating functions.
\newblock In {\em 21st IFAC World Congress}, pages 4262--4267, 2020.

\bibitem{GhafNRLK2023}
L.~Ghaffour, M.~Noack, J.~Reger, and T.-M. Laleg-Kirati.
\newblock Modulating functions approach for non-asymptotic state estimation of nonlinear {PDE}s.
\newblock In {\em 22nd IFAC World Congress}, pages 9875--9880, 2023.

\bibitem{JaZw12}
B.~Jacob and H.~Zwart.
\newblock {\em Linear Port-{H}amiltonian Systems on Infinite-dimensional Spaces}.
\newblock Number 223 in Operator Theory: Advances and Applications. Springer, Germany, 2012.

\bibitem{JoufR2015}
J.~{Jouffroy} and J.~{Reger}.
\newblock Finite-time simultaneous parameter and state estimation using modulating functions.
\newblock In {\em IEEE Conference on Control Applications}, pages 394--399, 2015.

\bibitem{KS2023}
K.~Kruse and C.~Seifert.
\newblock Final state observability estimates and cost-uniform approximate null-controllability for bi-continuous semigroups.
\newblock {\em Semigroup Forum}, 106:421--443, 2023.

\bibitem{LeRo95}
G.~Lebeau and L.~Robbiano.
\newblock Contr\^{o}le exact de l'\'{e}quation de la chaleur.
\newblock {\em Comm.\ Partial Differential Equations}, 20(1-2):335--356, 1995.

\bibitem{LinGu95}
Y.-J. Lin~Guo and W.~Littman.
\newblock Null boundary controllability for semilinear heat equations.
\newblock {\em Appl. Math. Optim.}, 32(3):281--316, 1995.

\bibitem{LiuLKPG2014}
D.~Y. Liu, T.-M. Laleg-Kirati, W.~Perruquetti, and O.~Gibaru.
\newblock Non-asymptotic state estimation for a class of linear time-varying systems with unknown inputs.
\newblock In {\em 19th IFAC World Congress}, pages 3732--3738, 2014.

\bibitem{PerdG1966}
F.~J. Perdreauville and R.~E. Goodson.
\newblock Identification of systems described by partial differential equations.
\newblock {\em Journal of Basic Engineering}, 88(2):463--468, June 1966.

\bibitem{PhReSc23}
F.~Philipp, T.~Reis, and M.~Schaller.
\newblock Infinite-dimensional port-{H}amiltonian systems -- a system node approach, 2023.
\newblock arXiv:2302.05168.

\bibitem{ReSc24}
T.~Reis and M.~Schaller.
\newblock Linear-quadratic optimal control for infinite-dimensional systems, 2024.
\newblock arXiv:2401.11302.

\bibitem{ReSc23a}
T.~Reis and M.~Schaller.
\newblock Port-{H}amiltonian formulation of {O}seen flows.
\newblock In F.~L. Schwenninger and M.~Waurick, editors, {\em Systems Theory and PDEs}, pages 123--148. Springer Nature Switzerland, Cham, 2024.

\bibitem{RojaNRPZ2022}
D.~Rojas, M.~Noack, J.~Reger, and G.~Pérez-Zuñiga.
\newblock State estimation for coupled reaction-diffusion {PDE} systems using modulating functions.
\newblock {\em Sensors}, 22:5008, 07 2022.

\bibitem{Rudi73}
W.~Rudin.
\newblock {\em Functional {A}nalysis}.
\newblock Mc{G}raw-Hill, New York, 1973.

\bibitem{Shin1954}
M.~Shinbrot.
\newblock On the analysis of linear and nonlinear dynamical systems from transient-response data.
\newblock {\em National Advisory Commitee For Aeronautics}, 1954.
\newblock Technical Note 3288.

\bibitem{Staffans2005}
O.~J. Staffans.
\newblock {\em Well-posed linear systems}, volume 103 of {\em Encyclopedia of Mathematics and Its Applications}.
\newblock Cambridge University Press, Cambridge, UK, 2005.

\bibitem{Tart07}
L.~Tartar.
\newblock {\em An Introduction to Sobolev Spaces and Interpolation Spaces}.
\newblock Title Lecture Notes of the Unione Matematica Italiana. Springer, Berlin, Heidelberg, 2007.

\bibitem{TuWe09}
M.~Tucsnak and G.~Weiss.
\newblock {\em Observation and Control for Operator Semigroups}.
\newblock Birkh\"auser Advanced Texts Basler Lehrb\"ucher. Birkh\"auser, Basel, 2009.

\end{thebibliography}
\end{document}